\documentclass[12pt,leqno,twoside]{article}

\usepackage{lineno}

\usepackage[version=4]{mhchem}

\usepackage[T1]{fontenc}
\usepackage[utf8]{inputenc}
\usepackage{amstext,amsmath,amscd, bezier,indentfirst,amsthm,amsgen,enumerate,bm}
\usepackage{todonotes}
\usepackage{subcaption}
\usepackage[all,knot,arc,import,poly]{xy}
\usepackage{amsfonts,color}
\usepackage{amssymb}
\usepackage{latexsym}
\usepackage{epsfig}
\usepackage{graphicx}
\usepackage{srcltx}
\usepackage{enumitem}
\usepackage{accents}
\usepackage{imakeidx}
\usepackage{rotating}
\usepackage{verbatim}
\usepackage{fancyhdr}
\usetikzlibrary{arrows}

\usepackage{arydshln}

\usepackage{tikz-cd}

\tikzset{
  > = {Latex},
  inner sep = 4pt,
  outer sep = auto
}

\usepackage{tikz}
\usetikzlibrary{patterns}
\usetikzlibrary{matrix, arrows,decorations.pathmorphing}
\usepackage{tikz-cd}
\tikzset{commutative diagrams/.cd}
\usepackage{framed,lipsum}
\usepackage{float}
\usepackage{pgfplots}
\usetikzlibrary{patterns}
\usetikzlibrary{intersections}
\usetikzlibrary{positioning, shadows,arrows}
\tikzstyle{every node}=[anchor=west, minimum height=3em]
\usetikzlibrary{shapes}
\usetikzlibrary{decorations.pathmorphing}

\usepackage{wrapfig}
\definecolor{forestgreen}{rgb}{0.00, 0.39, 0.00} 
\definecolor{blueblue}{rgb}{0.40, 0.00, 1.00} 
\definecolor{sienna}{rgb}{0.33, 0.08, 0.11}

\newcommand{\defi}{\textbf}

\newcommand{\p}{\partial}
\newcommand{\spn}{\text{span}}

\theoremstyle{plain}
\newtheorem{theorem}{Theorem}[section] 
\newtheorem*{theorem*}{Theorem}
\newtheorem*{hypothesis*}{Hypothesis}
\newtheorem{lemma}[theorem]{Lemma}
\newtheorem{proposition}[theorem]{Proposition}
\newtheorem{corollary}[theorem]{Corollary}
\newtheorem{notation}[theorem]{Notation}

\theoremstyle{definition}
\newtheorem{definition}[theorem]{Definition}
\newtheorem*{definition*}{Definition}
\theoremstyle{remark}
\newtheorem{remark}[theorem]{\sc Remark}
\newtheorem*{question*}{\sc Question}
\newtheorem*{remark*}{\sc Remark}
\newtheorem*{example*}{\sc Example}

\newtheorem{example}[theorem]{\sc Example}

\usepackage{mathscinet}
\usepackage[pdftex, pdfusetitle, plainpages=false,  bookmarks, bookmarksnumbered,colorlinks, linkcolor=blue, citecolor=red,filecolor=black, urlcolor=black]{hyperref}
\usepackage{pgfplots}
\pgfplotsset{width=7cm,compat=1.8}
\usepackage[tight]{minitoc} 
\usepackage{caption}
\makeatletter
\renewcommand*{\numberline}[1]{\hb@xt@1em{#1\hfil}} 
\makeatother

\usepackage{mhchem}
\usepackage{chemfig}

\usepackage[backend=bibtex,style=numeric,backref=true,maxbibnames=99]{biblatex}
\usepackage{pgfplots}

\pgfplotsset{width=7cm,compat=1.8}

\usepackage{caption}
\usepackage{titling}

\hypersetup{
    colorlinks=true,
    linkcolor=blue,
    filecolor=magenta,      
    urlcolor=violet,
    citecolor=red,
    pdftitle={Overleaf Example},
    pdfpagemode=FullScreen,
    }
\thanksmarkseries{arabic}

\newcommand\RR{\mathbb{R}}

\newcommand\by{\boldsymbol{y}}
\newcommand\byi{{\boldsymbol{y}_i}}
\newcommand\byj{{\boldsymbol{y}_j}}
\newcommand\byp{{\boldsymbol{y}^\prime}}

\newcommand\bk{\boldsymbol{k}}
\newcommand\bw{\boldsymbol{w}}

\newcommand\bx{\boldsymbol{x}}

\newcommand\bv{\boldsymbol{v}}
\newcommand\bz{\boldsymbol{z}}

\newcommand\bJ{\boldsymbol{J}}
\newcommand\bK{\boldsymbol{K}}

\newcommand\bX{\boldsymbol{X}}

\newcommand\bbeta{\boldsymbol\beta}
\newcommand\bB{\mathcal{B}}

\newcommand\blam{\boldsymbol{\lambda}}

\usepackage[a4paper]{geometry}
\title{The toric locus of a reaction network is a smooth manifold}

\date{}

\author{Gheorghe Craciun\thanks{University of Wisconsin-Madison, USA}, Jiaxin Jin\thanks{Ohio State University, USA}, Miruna-\c Stefana Sorea\thanks{SISSA (Scuola Internazionale Superiore di Studi Avanzati), Trieste, Italy and RCMA Lucian Blaga University, Sibiu, Romania}}

\begin{document}

\maketitle

\begin{abstract}

\noindent 
We show that the toric locus of a reaction network is a smoothly embedded submanifold of the Euclidean space. More precisely, we prove that the toric locus of a reaction network is the image of an embedding and it is diffeomorphic to the product space between the affine invariant polyhedron of the network and its set of complex-balanced flux vectors. Moreover, we prove that within each affine invariant polyhedron, the complex-balanced equilibrium depends smoothly on the parameters (i.e., reaction rate constants). We also show that the complex-balanced equilibrium depends smoothly on the initial conditions.
\end{abstract}

\tableofcontents

\newpage

\section{Introduction}

Reaction networks and interaction networks are very important in many different settings, such as chemistry, biochemistry, cell biology, population dynamics, the study of species dynamics in an ecosystem, or the study of the spread of infectious diseases. 

\medskip 

The most common type of mathematical models for analyzing the dynamics of concentrations or populations in these systems are {\em mass-action systems}, i.e., are based on the {\em principle of mass-action kinetics}, which simply says that the rate of each reaction or interaction is proportional to the concentrations of its interacting species \cite{feinberg, HJ,  cy}.

Mass-action systems are finite-dimensional dynamical systems with polynomial right-hand-sides. They can be rigorously defined by using one of the several (equivalent) mathematical representations of a reaction network; here we use the mathematical definition of reaction networks that is based on {\em Euclidean embedded graphs}, also called {\em E-graphs}, see Section \ref{sec:preliminaryNotions}. For other uses of Euclidean embedded graphs see also \cite{MR3920470}.


It may seem at first that the polynomial dynamical systems that result from mass-action kinetics are somehow special; it actually turns out that, except for the fact that they are restricted to the positive orthant\footnote{because the variables of interest are concentrations or populations, so they cannot take on negative values}, they can give rise to absolutely {\em all} the dynamics that fully general polynomial dynamical systems can exhibit on a compact set \cite{Brunner_Craciun_2018, cy}. In particular, mass-action systems can give rise to multiple equilibria, oscillations, and chaotic dynamics. 
For example, recall that the second part of Hilbert's 16th problem involves limit cycles of polynomial dynamical systems in the plane.  It is not hard to see that if this Hilbert problem could be solved for mass-action systems, then it would also immediately be solved in full generality, i.e., for all polynomial systems.

\medskip

Here we focus on the most studied class of mass-action systems, called {\em complex-balanced systems}. This class of dynamical systems has its origins in some important connections between reaction networks and  
thermodynamics; indeed, complex-balanced systems can be regarded as generalizations of finite-dimensional versions of the Boltzmann equation \cite{cy}.  In particular, the {\em complex balance condition} is a natural generalization of Boltzmann's {\em detailed balance condition}, which is a consequence of his principle of {microscopic reversibility} \cite{cy}.

Complex-balanced systems are known to be exceptionally stable: they have a unique positive equilibrium within each linear invariant subspace, and this equilibrium (which is called a {\em complex-balanced equilibrium}) is locally stable; furthermore, complex-balanced systems {\em cannot} give rise to oscillations or chaotic dynamics \cite{feinberg, HJ,  cy}. 
Moreover, it has actually been conjectured that any complex-balanced equilibrium is a {\em globally attracting point} within its linear invariant subspace, and this conjecture (called the {\em global attractor conjecture}) has already been proved in many cases \cite{Anderson, cdss2009, Craciun_Nazarov_Pantea_2013,  Gopalkrishnan_Miller_Shiu}. 


The complex balance condition also has very strong consequences for the dynamics of associated {\em stochastic} mass-action systems. In particular, it implies that such systems (if considered for the parameter values that satisfy complex balance) have a stationary distribution, which takes on an explicit form, as a product of Poisson distributions \cite{Anderson_Craciun_Kurtz_2010}.


Moreover, the same complex balance condition has very strong consequences for large classes of {\em reaction-diffusion systems}; for example, for systems without boundary equilibria, it can be shown that such (nonlinear) reaction-diffusion systems admit classical solutions, and these solutions converge exponentially fast to the complex-balanced equilibrium \cite{Desvillettes_Fellner_Tang_2017, Craciun_Jin_Pantea_Tudorascu_2021}

\subsection{Main results}

We consider a class of nonlinear dynamical systems, introduced by Horn and Jackson (see \cite{HJ}) and called {\em complex-balanced dynamical systems} or {\em toric dynamical systems}  (see \cite{cdss2009}). They are dynamical systems associated to reaction networks, under the assumption of mass-action kinetics, such that they can give rise to {\em complex-balanced equilibria} (see Definition~\ref{def:CB}). For a detailed introduction to these topics we refer the reader to \cite{feinberg, cy}. Given a reaction network, its \emph{toric locus} (see Definition \ref{def:VkG}) is the set of positive parameters giving rise to toric dynamical systems associated to this network. Recently, researchers have shown increasing interest in the study of the toric locus (see for instance \cite{haque2022disguised}, \cite{bcs}).

It was shown in \cite{Connected} that the toric locus is connected. Moreover, a construction in  \cite{Connected} exhibits a homeomorphism from a product space $\mathcal{P}$ to the toric locus; here  $\mathcal{P}$ denotes the product space between the affine invariant polyhedron of the reaction network and its set of
complex-balanced flux vectors. Using this homeomorphism, we construct a smooth embedding and prove that the toric locus is a smoothly embedded submanifold in the ambient space. In other words, we show that the toric locus is diffeomorphic to the product space $\mathcal{P}$ mentioned above, see Theorem \ref{thm:submanifold} and Corollary \ref{cor:diffeomorphicToTheProduct}. 
Moreover, we prove that the complex-balanced equilibria depend smoothly on the parameters belonging to the toric locus (Theorem \ref{thm:smoothDependanceEquil}), respectively on the initial condition (Theorem \ref{thm:smoothDependanceinitial}).

\medskip

We now state the main results obtained in this paper. 
Our first theorem shows that given a Euclidean embedded graph, the associated toric locus is a smooth manifold.

\begin{theorem}
\label{thm:submanifold}
Let $G=(V, E)$ be a weakly reversible Euclidean embedded graph. Then the toric locus $\mathcal{V}(G)$ is a smoothly embedded submanifold of $\mathbb{R}_{>0}^{\vert E \vert}$.
\end{theorem}

The second and third theorems show that
complex-balanced equilibria of the mass-action system vary smoothly with the parameters in the toric locus and with the initial conditions, respectively.

\begin{theorem} \label{thm:smoothDependanceEquil}
Let $G=(V, E)$ be a weakly reversible Euclidean embedded graph,  let $\mathcal{V}(G)$ be its toric locus, and $\mathcal{S}$ be its stoichiometric subspace. Consider also some positive  state ${\boldsymbol{x}_0} \in \mathbb{R}_{>0}^n$, and for any $\bk \in \mathcal{V}(G)$ denote by $\bx^*(\bk)$ the complex-balanced equilibrium of the mass-action system $(G, \bk)$ within the affine invariant subspace $({\boldsymbol{x}_0} + \mathcal{S})\cap \mathbb{R}_{>0}^n$. Then $\bx^*(\bk)$  depends smoothly on $\bk$. 
\end{theorem}

\medskip

\begin{theorem}
\label{thm:smoothDependanceinitial}
Let $G=(V, E)$ be a weakly reversible Euclidean embedded graph and fix some $\bk \in \mathcal{V}(G)$, where $\mathcal{V}(G)$ is the toric locus of $G$. 
For any $\bx_0 \in \RR^n_{>0}$ denote by $\bx^*(\bx_0)$ the complex-balanced equilibrium of the mass-action system $(G, \bk)$ within the affine invariant subspace $({\boldsymbol{x}_0} + \mathcal{S})\cap \mathbb{R}_{>0}^n$. Then $\bx^*(\bx_0)$  depends smoothly on $\bx_0$.
\end{theorem}

\subsection{Structure of the paper}

In Section \ref{sec:preliminaryNotions}, we present terminology and notations specific to dynamical systems arising from reaction networks, with an emphasis on mass-action complex-balanced dynamical systems. The latter are also referred to as toric systems. 
In Section \ref{sec:smooth_embedding}, we prove our first main result which says that the toric locus is a smoothly embedded submanifold of Euclidean space. 
Based on this result, in Section \ref{sec:smooth_dependent_toric_locus} we show that the complex-balanced equilibrium depends smoothly on the parameter values in the toric locus.
Finally, in Section \ref{sec:smooth_dependent_initial}, we prove Theorem \ref{thm:smoothDependanceinitial}
showing that the complex-balanced equilibrium depends smoothly on the initial conditions.

\section{Preliminaries} \label{sec:preliminaryNotions}

\noindent
In this preliminary section, we present classical and relevant terminology and notations from the field of reaction network theory. We mostly follow the presentation from \cite{Connected} (see also \cite{cy}). More precisely, we recall  
notions and results about deterministic dynamical systems generated by reaction networks, under the assumption of mass-action kinetics, with a focus on complex-balanced systems. Let us first begin with some useful notations.

\begin{notation}

(a) Let $\mathbb{R}_{\geq 0}^n$, respectively $\mathbb{R}_{>0}^n$ denote the sets of vectors with non-negative, respectively positive real components. Similarly, $\mathbb{Z}_{\geq 0}^n$ denotes the set of vectors with non-negative integer components. The cardinality of a set $S$ is denoted by $\vert S \vert$. Let $M \sqcup N$ denote the disjoint union of the sets $M$ and $N$.

(b) Consider two vectors $\bx = (x_1, \ldots, x_n)^{\intercal} \in \mathbb{R}_{>0}^n$ and $\by = (y_1, \ldots, y_n)^{\intercal} \in \mathbb{R}_{\ge 0}^n$. 
Throughout this paper we will use the  following multivariate notations:

\begin{equation} \notag
\begin{split}
\bx \circ \by & := (x_1 y_1, \ldots, x_n y_n)^{\intercal},
\\ {\bx}^{\by} & := x_1^{y_{1}} x_2^{y_{2}} \ldots x_n^{y_{n}},
\\ \log (\bx) & := (\log(x_1), \ldots, \log(x_n))^{\intercal},
\\ \exp (\bx) & := (\exp(x_1), \ldots, \exp(x_n))^{\intercal}.
\end{split}
\end{equation}

\end{notation}

\subsection{Mass-action dynamical systems and reaction networks}
Classically, in the definition of a reaction network, there are some important sets: the set of species, the set of complexes, and the set of reactions (see, for instance, \cite[Definition 3.1.1.]{feinberg}). In this paper, we define a reaction network to be a directed graph embedded in the Euclidean
space, as it was first presented in  \cite{MR3920470}. Due to the combinatorial and geometric aspects of this way of defining a reaction network, this equivalent definition has appeared frequently in the recent research literature (see for example \cite{cy} and references therein),  since it is very advantageous for algebraic computations.

\begin{definition}

\begin{enumerate}

\item 
Let $n$ denote the \textbf{number of species} involved in the reaction network. The \textbf{species} are denoted by $X_1,\ldots, X_n$.
\item 
The \defi{concentration} of the species $X_i$, for $i=1,\ldots, n$ is denoted by $x_i=x_i(t)$. Note that $x_i$ is a  function of time $t$. At any fixed time $t \geq 0$, we obtain a vector $\bx(t) = (x_1(t), \ldots, x_n(t))^\intercal \in \mathbb{R}^n$, that we call a \defi{state} of the system at time $t$.
\item 
A \defi{complex} is a formal linear combination of species $\{ X_i\}^n_{i=1}$, having non-negative real coefficients. A directed edge between two distinct complexes is called a \defi{reaction}.
\end{enumerate}
\end{definition}

\begin{definition}[{\cite{MR3920470}}]
\label{def:reactNetw_RateConstants}
Consider a finite directed graph $G=(V, E)$. We say that $G$ is a \defi{reaction network} (or a \defi{Euclidean embedded graph}) if it has a finite set of \defi{vertices} $V\subset \mathbb{R}^n$ and a finite set of \defi{edges} $E\subseteq V\times V$. Throughout this paper, we suppose that $G=(V, E)$ does not have isolated vertices, nor self-loops.
\begin{enumerate}
\item 
Let $m$ denote the \textbf{number of vertices}, and let $V = \{ \boldsymbol{y}_1, \ldots, \boldsymbol{y}_m\}$ be the set of vertices. There is a one-to-one correspondence between the vertices and the complexes of the reaction network, namely each $\byi \in V$ corresponds to a unique complex: the entries of the vertex $\byi \in V$ are exactly the corresponding coefficients of the associated formal linear combination.
\item 
We represent a reaction in the network by a directed edge in the graph,  connecting two vertices $\byi\in V$ to $\byj\in V$. We denote it by $\byi \rightarrow \byj\in E$. The difference vector $\byj - \byi\in\mathbb{R}^n$ is called a \defi{reaction vector}. Here $\byi$, respectively $\byj$ denote the \defi{source vertex}, respectively the \defi{target vertex}.
\end{enumerate}
\end{definition}

\begin{definition} \label{def:weaklyRev}
Consider a Euclidean embedded graph, $G=(V, E)$.
\begin{enumerate}
\item 
The connected components (or \defi{linkage classes}) of $G=(V, E)$ partition the set of vertices $V$. 
Denote by $V = V_1 \sqcup V_2 \cdots \sqcup V_\ell$, where each $V_i$ represents a connected component of $G$ and the \defi{number of connected components} is $\ell$. 
\item 
If every edge of a connected component is part of an oriented cycle, then we say that the component is \defi{strongly connected}. 
\item 
If every connected component of a graph $G=(V, E)$ is strongly connected, we say that $G$ is \defi{weakly reversible}. 
\end{enumerate}
\end{definition}

If we work in the context of mass-action kinetics (see \cite[page 28]{feinberg}), we obtain the ODE system of the form (\ref{eq:massAction}) (see Definition \ref{def:massAction} below). This is due to the fact that, under this assumption, the reaction rates are proportional to the product of the concentrations of the reactant species.

\begin{definition} \label{def:massAction}

Let $G=(V, E)$ be a Euclidean embedded graph. Under mass-action kinetics, we decorate each edge $\byi\rightarrow \byj$ with a positive \defi{reaction rate constant}, denoted by $k_{ij}$ or $k_{\byi \rightarrow \byj}$. 
Then the \defi{vector of reaction rate constants} is ${\bk}:=(k_{ij})\in\mathbb{R}_{>0}^{E}$ and the  \defi{mass-action system associated to} $(G,{\bk})$ 
on $\RR^n_{>0}$ is the following:
\begin{equation} \label{eq:massAction}
\frac{\mathrm{d}\bx}{\mathrm{d} t}= \sum_{\byi \rightarrow \byj \in E}k_{\byi \rightarrow \byj} \bx^\byi(\byj -\byi).
\end{equation}
\end{definition}

\begin{definition} \label{def:stoichiom}

Consider a Euclidean embedded graph $G = (V, E)$. Then the \defi{stoichiometric subspace} of $G$ is the following:
\begin{equation}
\mathcal{S} = \spn \{ \byj - \byi: \byi \rightarrow \byj \in E \}.
\end{equation}
It is known that any solution to \eqref{eq:massAction}
with initial condition ${\boldsymbol{x}_0}\in \mathbb{R}_{>0}^n$ and $V \subset \mathbb{Z}_{\geq 0}^n$, is confined to $({\boldsymbol{x}_0} + \mathcal{S})\cap \mathbb{R}_{>0}^n$, see \cite{cy}.
Hence, we say that $({\boldsymbol{x}_0} + \mathcal{S})\cap \mathbb{R}_{>0}^n$ is the \defi{affine invariant polyhedron} of ${\boldsymbol{x}_0}$. 
In the sequel, we use the shorter notation: 
$\mathcal{S}_{{\boldsymbol{x}_0}} := ({\boldsymbol{x}_0} + \mathcal{S})\cap \mathbb{R}_{>0}^n$.
\end{definition}

\subsection{The toric locus of a reaction network}

\noindent
In what follows we recall the definition of \emph{complex-balanced dynamical systems}
and of the associated \emph{toric locus}, which are key concepts in this paper.

\begin{definition} \label{def:CB}
Consider the mass-action system $(G, {\bk})$ given in (\ref{eq:massAction}). 
A state ${\boldsymbol{x}^*} \in \mathbb{R}_{>0}^n$ satisfying 
\begin{equation} \label{eq:ss}
\frac{\mathrm{d}\bx}{\mathrm{d} t}= \sum_{\byi \rightarrow \byj \in E}k_{\byi \rightarrow \byj} ({\boldsymbol{x}^*})^\byi(\byj -\byi) = \mathbf{0}
\end{equation}
is called a \defi{positive steady state}. 
Moreover, if at each vertex $\boldsymbol{y}_0 \in V$ we have 
\begin{equation} \label{eq:cB}
\sum_{\by_0 \to \byp \in E} k_{\by_0 \to \byp} ({\boldsymbol{x}^*})^{\by_0}
= \sum_{\by \to \by_0 \in E}k_{\by \to \by_0}
({\boldsymbol{x}^*})^{\by},
\end{equation} then the positive steady state ${\boldsymbol{x}^*} \in \mathbb{R}_{>0}^n$ is called a  \defi{complex-balanced steady state}.

Given a mass-action system $(G, \bk)$, we say that it is a \defi{complex-balanced system} (also called \defi{toric dynamical system} \cite{cdss2009}), if every steady state of the system is a complex-balanced steady state.
\end{definition}

Definition \ref{def:VkG} below was introduced in \cite[Definition 2.2]{bcs} (see also \cite[Definition 2.14]{Connected}) and it was due to the results in \cite{cdss2009}.

\begin{definition} \label{def:VkG}
Consider a Euclidean embedded graph $G=(V, E)$. The set of parameters ${\bk}\in\mathbb{R}_{>0}^{|E|}$, for which the mass-action system $(G, \bk)$ is toric or complex-balanced is called the \defi{toric locus} of the mass-action system given by the Euclidean embedded graph $G$.
We denote the toric locus by $\mathcal{V}(G) \subseteq \mathbb{R}_{>0}^{|E|}$.
\end{definition}

Besides the nice algebraic and combinatorial structure of toric dynamical systems, it is also well-known that they have strong stability properties that are desirable in many applications, as the following classical theorem has shown.  

\begin{theorem}[\cite{HJ}]
\label{thm:HJ}

Let $(G, \bk)$ be a complex-balanced dynamical system, and let $\mathcal{S}$ be its associated stoichiometric subspace. 
Denote by $\bx^* \in \RR^n_{>0}$ a steady state of $(G, \bk)$.
Then:
\begin{enumerate}[label=(\alph*)]
\item All the positive steady states are complex-balanced and there is exactly one steady state within each invariant polyhedron. 

\item For any complex-balanced steady state $\bx$ we have: $\ln \bx  \in \ln \bx^* + \mathcal{S}^\perp$.  

\item Every complex-balanced steady state is locally asymptotically stable within its invariant polyhedron. 
\end{enumerate}
\end{theorem}

In the sequel we only work in the context of $\mathcal{V}(G)$ being non-empty. Thus we assume that $G=(V, E)$ is a weakly reversible graph throughout the paper.

\section{The toric locus is a smoothly embedded submanifold}
\label{sec:smooth_embedding}

This section aims to prove Theorem \ref{thm:submanifold}.
For this purpose, we define the \emph{set of complex-balanced flux vectors} (see  Definition \ref{def:fluxVectors} and \cite[Section 4.1]{Connected}) and two maps in Theorem \ref{thm:mapTildePhi} and Definition \ref{def:mapTildeTildePhi} below.

\subsection{The set of complex-balanced flux vectors \texorpdfstring{$\mathcal{B}(G)$}{Bg}}
\label{subsec:CBFluxVectors}

\begin{definition}
\label{def:fluxVectors}

Let $G=(V, E)$ be a Euclidean embedded graph. 
\begin{enumerate}
\item[(a)] A vector $\bbeta = (\beta_{\byi \to \byj})_{\byi \to \byj \in E}$ \defi{satisfies the complex-balanced condition} on $G$, if at each vertex $\by_0 \in V$ we have
\begin{equation}
\label{def: CBF}
\sum_{\by \to \by_0 \in E} \beta_{\by \to \by_0} 
= \sum_{\by_0 \to \byp \in E} \beta_{\by_0 \to \byp}.
\end{equation} 
We denote the set of all vectors satisfying condition \eqref{def: CBF} by 
\begin{equation} \notag
\tilde{\bB} (G) :=
\{\bbeta \in \RR^{E} \mid \bbeta  \text{ satisfies the complex-balanced condition on $G$}\}.
\end{equation}

\item[(b)] Further, a vector $\bbeta = (\beta_{\byi \to \byj})_{\byi \to \byj \in E}$ is called a \defi{complex-balanced flux vector}, if $\bbeta \in \RR^{|E|}_{>0}$ and it satisfies the complex-balanced condition on $G$.
We denote the set of all complex-balanced flux vectors by 
\begin{equation} \notag
\mathcal{B}(G) := \tilde{\bB} (G) \cap \RR^{|E|}_{>0}.
\end{equation}
\end{enumerate}
\end{definition}

\begin{lemma}(\cite[Lemma 4.4]{Connected})
\label{lem:weak_reversible_BG}
Let $G=(V, E)$ be a Euclidean embedded graph. Then $\bB (G) \neq \emptyset$ if and only if $G=(V, E)$ is weakly reversible.
\end{lemma}

Suppose $G = (V, E)$ is a weakly reversible Euclidean embedded graph. Lemma \ref{lem:weak_reversible_BG} shows that both sets $\bB (G)$ and $\tilde{\bB} (G)$ are non-empty. Since the set $\tilde{\bB} (G)$ is formed by finitely many linear restrictions given in \eqref{def: CBF}, the set $\tilde{\bB} (G)$ is a linear subspace of $\mathbb{R}^{|E|}$.
Note that the set $\mathcal{B}(G)$ is obtained by intersecting $\tilde{\bB} (G)$ with the positive orthant $\mathbb{R}_{>0}^{E}$, therefore $\mathcal{B}(G)$ forms a pointed cone in $\mathbb{R}_{>0}^{|E|}$ (see \cite[Lemma 4.5]{Connected}).

\begin{lemma}
\label{lem:constant_weight_flux}
Let $G=(V, E)$ be a weakly reversible Euclidean embedded graph and let $\bbeta = (\beta_{\by \rightarrow \by'})_{\by \rightarrow \by' \in E}$ be a (non-zero) complex-balanced flux vector in $\bB (G)$.
Suppose for every vertex $\by \in V$, there exists a corresponding weight $c(\by) \in \RR$, such that
\begin{equation} \label{tilde_beta}
\widetilde{\beta}_{\by \rightarrow \by'} = c(\by) \beta_{\by \rightarrow \by'}.
\end{equation}
If $\widetilde{\bbeta} = (\widetilde{\beta}_{\by \rightarrow \by'})_{\by \rightarrow \by' \in E} \in \tilde{\bB} (G)$, then the weights $c ( \cdot )$ are constant within each connected component of $G$, i.e., for any two vertices $\by, \by'$ in the same connected component of $G$, $c(\by) = c (\by')$.
\end{lemma}

\begin{proof}
First, we consider the case when $G=(V, E)$ has a single connected component. 
In this case, it is equivalent to showing that $c(\by)$ is a constant function with respect to any vertex $\by \in V$.

Let us argue by contradiction. 
 Suppose that there exist two distinct vertices $\by, \by' \in V$, such that $c(\by)\neq c(\by')$.
Assume that $M = \max\limits_{\by \in V} c(\by)$, we define the following set as
\begin{equation}
V_M = \{ \by \in V \mid c (\by) = M \}.
\end{equation}
We note that $\emptyset \neq V_M \subsetneq V$. Since $G=(V, E)$ is weakly reversible, there are some reactions between vertices in $V_M$ and vertices in $V \setminus V_M$. Thus we can find two vertices $\by_0 \in V_M$ and $\hat{\by} \in V \setminus V_M$, such that 
\begin{equation} \label{vertex y0}
M = c(\by_0) > c(\hat{\by})
\ \text{ and } \
\hat{\by} \rightarrow \by_0 \in E.
\end{equation} 
Since $\widetilde{\bbeta} \in \widetilde{\bB}(G)$, for each vertex $\by \in V$ we have
\begin{equation} \label{eq:tilde_beta_1}
\sum_{\by \rightarrow \by' \in E} \widetilde{\beta}_{\by \rightarrow \by'}
= \sum_{\by''\rightarrow \by \in E} \widetilde{\beta}_{\by'' \rightarrow \by}.
\end{equation}
From $\bbeta\in \bB(G)$ and \eqref{tilde_beta}, we deduce
\begin{equation} \label{eq:tilde_beta_2}
\sum_{\by_0\rightarrow \by' \in E}\widetilde{\beta}_{\by_0\rightarrow \by'}
= c(\by_0) \sum_{\by_0\rightarrow \by' \in E} \beta_{\by_0\rightarrow \by'}
= c(\by_0) \sum_{\by'' \rightarrow \by_0 \in E} \beta_{\by'' \rightarrow \by_0}.
\end{equation}
Using \eqref{vertex y0} and \eqref{eq:tilde_beta_2}, we obtain
\begin{equation} \notag
\sum_{\by_0\rightarrow \by' \in E}\widetilde{\beta}_{\by_0\rightarrow \by'} 
= \sum_{\by'' \rightarrow \by_0 \in E} c(\by_0) \beta_{\by'' \rightarrow \by_0}
> \sum_{\by'' \rightarrow \by_0 \in E} c(\by'') \beta_{\mathbf{y''}\rightarrow \by_0}
= \sum_{\by'' \rightarrow \by_0 \in E}\widetilde{\beta}_{\by'' \rightarrow \by_0},
\end{equation}
which contradicts with \eqref{eq:tilde_beta_1}. This concludes the proof in the case of a single connected component.

\medskip

Next, we consider the case when $G=(V, E)$ has $\ell > 1$ connected components $\{ V_i \}^{\ell}_{i=1}$.
We claim that $c(\by)$ is a constant function with respect to vertices in the same connected component.

Let us argue by contradiction.  Suppose that there exists a connected component $V_1$ with two distinct vertices $\by, \by' \in V_1$, such that $c(\by) \neq c(\by')$.
Follow the steps in the single connected component case and assume $M = \max\limits_{\by \in V_1} c(\by)$. Then we can find two vertices $\by_0, \hat{\by} \in V_1$, such that 
\begin{equation} \label{vertex y0_L1}
M = c(\by_0) > c(\hat{\by})
\ \text{ and } \
\hat{\by} \rightarrow \by_0 \in E.
\end{equation}
Since $\widetilde{\bbeta} \in \widetilde{\bB}(G)$, by Definition \ref{def:fluxVectors}, for each vertex $\by \in V$ we have:
\begin{equation} \label{tilde B G L1}
\sum_{\by \rightarrow \by' \in E} \widetilde{\beta}_{\by \rightarrow \by'}
= \sum_{\by'' \rightarrow \by \in E} \widetilde{\beta}_{\by'' \rightarrow \by}.
\end{equation}
Using $\bbeta\in \bB(G)$, \eqref{tilde_beta} and \eqref{vertex y0_L1}, we have
\begin{equation} \notag
\sum_{\by_0\rightarrow \by' \in E}\widetilde{\beta}_{\by_0\rightarrow \by'} 
 = c(\by_0) \sum_{\by_0\rightarrow \by' \in E} \beta_{\by_0\rightarrow \by'}
= c(\by_0) \sum_{\by \rightarrow \by_0 \in E} \beta_{\by \rightarrow \by_0}.
\end{equation}
Note that $\by \rightarrow \by_0 \in E$ if and only if $\by \in V_1$. Thus we derive
\begin{equation} \notag
\sum_{\by_0\rightarrow \by' \in E}\widetilde{\beta}_{\by_0\rightarrow \by'} 
> \sum_{\by \rightarrow \by_0 \in E} c(\by) \beta_{\by \rightarrow \by_0 \in E}
=\sum_{\by \rightarrow \by_0} \widetilde{\beta}_{\by \rightarrow \by_0},
\end{equation}
which contradicts \eqref{tilde B G L1}. 
Therefore, we conclude this lemma.
\end{proof}

\subsection{A smooth map to the toric locus}

We start by defining a map $\varphi$ from the product space $\mathcal{S}_{\bx_0} \times \bB (G)$ to the toric locus $\mathcal{V}(G)$, which was first introduced in \cite[Definition 4.9]{Connected}. 

\begin{theorem}(\cite[Theorem 4.8]{Connected})
\label{thm:mapTildePhi}
Let $G = (V, E)$ be a weakly reversible Euclidean embedded graph. Then for any $\bx_0\in\RR^n_{>0}$, there exists a topological embedding: 
\begin{equation}
\varphi: \mathcal{S}_{\bx_0} \times \bB (G) \rightarrow \RR^{|E|}
\end{equation}
such that for $(\bx, \bbeta) \in \mathcal{S}_{\bx_0} \times \bB (G)$,
\begin{equation} \label{eq:mapPhi}
    \varphi(\bx, \bbeta):=\bigg (\frac{\beta_{\by \rightarrow \by'}}{\bx^{\by}}\bigg )_{\by \rightarrow \by' \in E}.
\end{equation}
Moreover, $\varphi$ yields a homeomorphism between
the product space $\mathcal{S}_{\bx_0} \times \bB (G)$ and the toric locus $\mathcal{V}(G)$.
\end{theorem}

Note that in \cite[Definition 4.9]{Connected} the target of the map $\varphi$ was $\mathcal{V}(G)$, whereas in Theorem \ref{thm:mapTildePhi} above we use  $\RR^{|E|}$. 

Recall that our goal is to prove Theorem \ref{thm:submanifold}. To this end, we will show in Proposition \ref{lemma:smoothImmersion} that the map $\varphi$ is a smooth immersion. Then by Theorem \ref{thm:mapTildePhi} we will conclude that $\varphi$ is a smooth embedding, in Proposition \ref{prop:PhiIsEmbedding}.

Next, let us introduce the map $\hat{\varphi}$, which is an extension of the map $\varphi$ from Theorem \ref{thm:mapTildePhi} and it will be used in the proof of Proposition \ref{lemma:smoothImmersion}.

\begin{definition} \label{def:mapTildeTildePhi}
Let $G = (V, E)$ be a weakly reversible Euclidean embedded graph. Define the following map: 
\begin{equation}
\hat{\varphi} : \RR^n_{>0} \times \RR^{|E|} \rightarrow \RR^{|E|}
\end{equation} 
such that for $(\bx, \bbeta)\in \RR^n_{>0} \times \RR^{|E|}$,
\begin{equation} \label{hat phi} 
\hat{\varphi} (\bx, \bbeta)
:= \bigg (\frac{\beta_{\by \rightarrow \by'}}{\bx^{\by}}\bigg )_{\by \rightarrow \by' \in E}.
\end{equation}
\end{definition}

\begin{remark} \label{rmk:extension_map}
Note that $\mathcal{S}_{\bx_0} \times \bB (G) \subset \RR^n_{>0} \times \RR^{|E|}$, thus we have $\text{Dom} (\varphi) \subset \text{Dom} (\hat{\varphi})$. It is also clear that $\hat{\varphi} \big|_{\mathcal{S}_{\bx_0} \times \bB (G)} = \varphi$, therefore $\hat{\varphi}$ is an extension of $\varphi$.
\end{remark}

\begin{example}\label{ex:Segre}
The toric locus of the reversible Euclidean embedded graph (reaction network) depicted in Figure \ref{fig:4thchamber} is the non-negative real part of the {\em Segre variety}, given by the equation $k_{34}k_{21}=k_{43}k_{12}$, $k_{ij}\in\mathbb{R}$, $k_{ij}>0$; see \cite{shiu10} and \cite[Section 5]{bcs}. The Segre variety is the classical well-known embedding of the product of the projective line with itself in $\mathbb{P}^3$. It is a {\em smooth} quadric (a determinantal surface); see \cite[page 478]{MR1288523}. Note that Segre embeddings appear frequently in Applied Algebraic Geometry, for example in Algebraic Statistics (see the independence model for two binary
random variables in \cite[Example 8.27]{MR4423369}) and in Quantum Mechanics and Quantum Information Theory (see \cite{gharahi2020fine}).
\end{example}

\begin{figure}[H]
    \centering
   \begin{tikzpicture}
\draw[semithick,lightgray] (-0.1,-0.1) grid (3.6,3.6);
      \draw[-{Stealth},semithick] (-0.2,0) -- (3.6,0) node[below] {};
      \draw[-{Stealth},semithick] (0,-0.2) -- (0,3.6) node[left] {};
      \node[label=right:{$\boldsymbol{y}_1$}] at (2.6,-0.37) {};
    \node[label=right:{}](M) at (2.85,0) {};
      \node[label=right:{$\boldsymbol{y}_2$}] at (1.9,1.15) {};
      \node[label=right:{}](N) at (1.85,1) {};
      \node[label=right:{$\boldsymbol{y}_3$}] at (0.85,2.2) {};
      \node[label=right:{}](P) at (0.85,2) {};
      \node[label=left:{$\boldsymbol{y}_4$}](Q) at (-0.15,3) {};
      \draw[fill] (M) circle (1.5pt);
      \draw[fill] (N) circle (1.5pt);
      \draw[fill] (P) circle (1.5pt);
      \draw[fill] (Q) circle (1.5pt);
      \def\shf{0.2ex};
      \draw[->,transform canvas={xshift=\shf,yshift=\shf},thick, red] (M) -- (N);
      \draw[<-,transform canvas={xshift=-\shf,yshift=-\shf},thick, red] (M) -- (N);
      \def\shf{0.2ex};
      \draw[->,transform canvas={xshift=\shf,yshift=\shf},thick, red] (P) -- (Q);
      \draw[<-,transform canvas={xshift=-\shf,yshift=-\shf},thick, red] (P) -- (Q);
\end{tikzpicture}
    \caption{The Euclidean embedded graph $G$ (i.e., reaction network) from Example \ref{ex:Segre}. The toric locus of the mass-action system given by $G$ is the non-negative real part of the {\em Segre variety}, given by the polynomial equation $k_{34}k_{21}=k_{43}k_{12}$. In Figure \ref{fig:segreAffine} we depict the surface from the affine chart $k_{43}=1$.}
    \label{fig:4thchamber}
\end{figure}
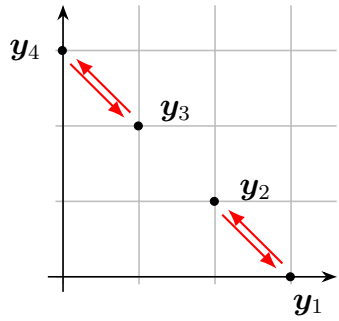
\begin{figure}[H]
    \centering
    \includegraphics[height=6cm]{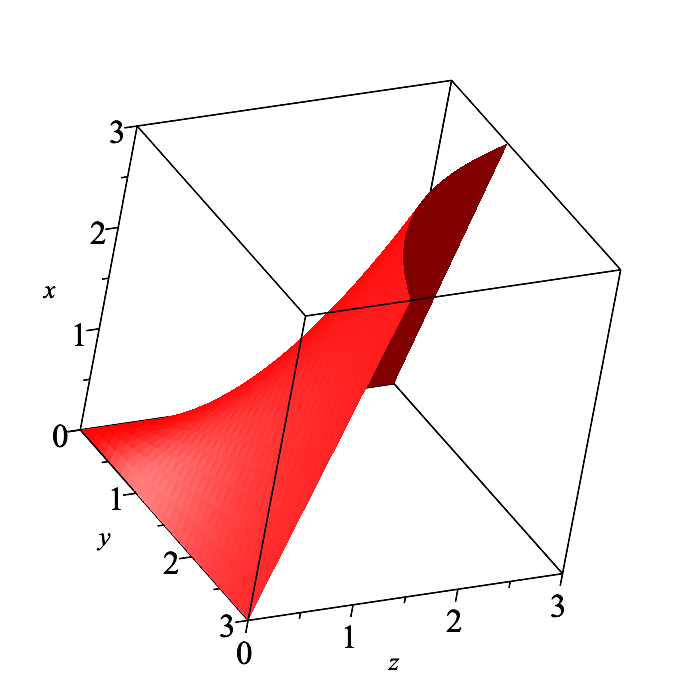}
    \includegraphics[height=6cm]{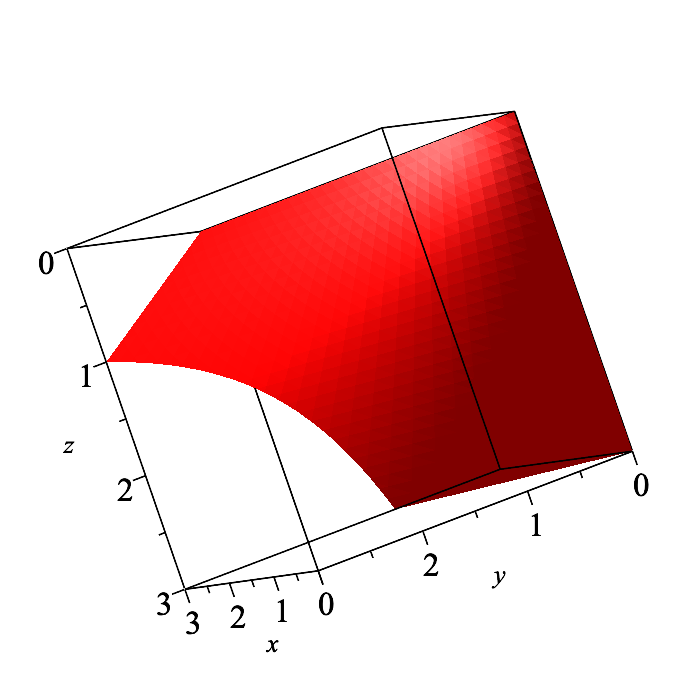}
    \caption{The toric locus of the mass-action system given by the Euclidean embedded graph $G$ from Figure \ref{fig:4thchamber}. See Example \ref{ex:Segre}.}
    \label{fig:segreAffine}
\end{figure}

\subsection{Proof of Theorem \ref{thm:submanifold}}
We first recall some definitions in differential topology, which appear in the main theorems.

\begin{definition}
\label{def:embedding} 
Given two smooth manifolds $A$ and $B$, let $f: A \to B$ be a smooth map. Then $f$ is an \defi{immersion} if its derivative is everywhere injective. See \cite[page 77]{leeSmooth}. 
A \defi{topological embedding} is a homeomorphism onto its image. See \cite[page 54]{leeTopological}. 
A \defi{smooth embedding} is an immersion which is also a topological embedding. Further, a \defi{(smoothly) embedded submanifold} is the image of a smooth embedding. See \cite[page 85]{leeSmooth}.
\end{definition}

For a textbook in differential topology we refer the reader to \cite{leeSmooth} (see also \cite{MR2680546}).

\medskip

The following proposition is the most important step towards proving Theorem \ref{thm:submanifold}. 

\begin{proposition} 
\label{lemma:smoothImmersion}
The map $\varphi$ from Theorem \ref{thm:mapTildePhi} is a smooth immersion.
\end{proposition} 

\begin{proof}

For simplicity, throughout the proof we use the following notation:
\[
M := \text{Dom} (\varphi) 
= \mathcal{S}_{\bx_0}\times \bB(G)
\]
Recall that $\mathcal{S}_{\bx_0} = ({\boldsymbol{x}_0} + \mathcal{S})\cap \mathbb{R}_{>0}^n$ and $\mathcal{B}(G) = \tilde{\bB} (G) \cap \RR^{|E|}_{>0}$.
Since $\mathcal{S}$ and $\widetilde{\bB}(G)$ are linear subspaces of $\RR^n$ and $\RR^{|E|}$ respectively, we derive that both $\mathcal{S}_{\bx_0}$ and $\mathcal{B}(G)$ are smooth manifolds, and thus $M$ is a smooth manifold. 
Note that for any
$\bx \in \mathcal{S}_{\bx_0} \subset \mathbb{R}_{>0}^n$ and
$\by \in \mathbb{R}_{\geq 0}^n$, we obtain $\bx^{\by} > 0$.
Hence $\varphi$ is a smooth map on $M$.

By Definition \ref{def:embedding}, it suffices to show that the derivative of 
$\varphi$ is injective. 
Since both $\mathcal{S}$ and $\widetilde{\bB}(G)$ are linear subspaces, the tangent space of $M$ at the point $(\bx, \bbeta) \in M$ is
\begin{equation} \label{tangent space}
\mathrm{T}_{(\bx, \bbeta)} M 
= \mathcal{S}\times \widetilde{\bB}(G).
\end{equation}
Note that both $\mathcal{S}$ and $\widetilde{\bB}(G)$ are independent with respect to $\bx$ and $\bbeta$, hence we let $\mathcal{T} := \mathrm{T}_{(\bx, \bbeta)} M$ denote the tangent space of $M$ at any point.
Then we consider the differential of $\varphi$ at the point $(\bx, \bbeta) \in M$ as follows:
\begin{equation}
\mathrm{d} \varphi_{(\bx, \bbeta)}:
\mathcal{T} \rightarrow \RR^{|E|}.
\end{equation}

We claim $\mathrm{d} \varphi$ at any point in $M$ has only trivial kernel in $\mathcal{T}$, that is, for any $(\bx, \bbeta) \in M$
\begin{equation}\label{eq:ker}
\mathrm{ker}(\mathrm{d} \varphi_{(\bx, \bbeta)}) \cap \mathcal{T}
= \{ \mathbf{0} \}.
\end{equation}

Assuming the claim, we deduce that the kernel of $\mathrm{d} \varphi$ at any point in $M$ only contains the zero vector, thus the derivative of $\varphi$ is injective. Since $\varphi$ is also smooth, by Definition \ref{def:embedding} we conclude that $\varphi$ is a smooth immersion. Thus it remains to prove the claim (\ref{eq:ker}).

\medskip

To simplify our computations, we consider $\hat{\varphi}$ in Definition \ref{def:mapTildeTildePhi} and work with the differential of $\hat{\varphi}$, such that for any $(\bx, \bbeta) \in M$
\begin{equation}
\mathrm{d} \hat{\varphi}_{(\bx, \bbeta)} :
\mathbb{R}^n \times \RR^{|E|}\rightarrow \RR^{|E|}.
\end{equation}
From Remark \ref{rmk:extension_map}, $\hat{\varphi}$ is an extension of $\varphi$ with $\hat{\varphi} \big|_{M} = \varphi$, which implies
$\mathrm{d} \hat{\varphi}\big|_{\mathcal{T}} = \mathrm{d} \varphi$ at any point in $M$. 
Thus we only need to prove that for any $(\bx, \bbeta) \in M$ and any $\bv \in \mathcal{T} \ \backslash \ \{\mathbf{0} \}$,
\begin{equation} \notag
\mathrm{d} \hat{\varphi}_{(\bx, \bbeta)} (\bv) \neq \mathbf{0}.
\end{equation}

By contradiction, suppose there exists some 
$\bv \in \mathcal{T} \ \backslash \ \{\mathbf{0} \}$ such that 
$\mathrm{d} \hat{\varphi} (\bv) = \mathbf{0}$. 
Here we write $\bv$ as 
\begin{equation} \label{eq:v=tvtb}
\bv = (\tilde{\bv}, \widetilde{\bbeta}) \in \mathcal{T} =\mathcal{S} \times \widetilde{\bB}(G),
\end{equation}
and we denote 
$\tilde{\bv} = (v_1, \ldots, v_n) \in \mathcal{S}$ and $\widetilde{\bbeta} = (\widetilde{\beta}_1, \ldots, \widetilde{\beta}_{\vert E\vert}) \in \widetilde{\bB}(G)$.

Let us focus on $\mathrm{d} \hat{\varphi}_{(\bx, \bbeta)}$, that is the Jacobian matrix of
$\hat{\varphi}$ at $(\bx, \bbeta)$.
By Definition \ref{def:mapTildeTildePhi}, the Jacobian matrix is given by
\begin{equation}
\bJ_{\hat{\varphi}} (\bx, \bbeta) 
= \left[ \frac{\partial \hat{\varphi}}{\partial \bx}, \frac{\partial \hat{\varphi}}{\partial \bbeta} \right].
\end{equation}
Consider a fixed reaction 
$\by \rightarrow \by' \in E$ and let us  compute the corresponding row vector 
\begin{equation} \label{eq:yy'row}
\left( \frac{\partial \hat{\varphi}_{\by \rightarrow \by'}}{\partial \bx}, \
\frac{\partial \hat{\varphi}_{\by \rightarrow \by'}}{\partial \bbeta} \right) 
\end{equation}
in the Jacobian matrix $\bJ_{\hat{\varphi}} (\bx, \bbeta)$.
For the first $n$ entries in the row given by \eqref{eq:yy'row}, which correspond to the derivatives of $\hat{\varphi}_{\by\longrightarrow \by'}$  with respect to $\bx$, we obtain
\begin{equation} \label{diff first n}
\frac{\partial \hat{\varphi}_{\by\longrightarrow \by'}}{\partial x_j}
= \frac{\beta_{\by \rightarrow \by'}}{\bx^{\by}}\frac{-y_j}{x_j}
\ \text{ for } j = 1, \ldots, n,
\end{equation}
where $\bx = (x_1, \ldots, x_n) \in \RR^n_{>0}$ and $\by = (y_1, \ldots, y_n)  \in \RR^n_{\geq 0}$.
Furthermore, the last $\vert E \vert$ entries of the row in \eqref{eq:yy'row} are obtained by taking the derivatives of $\hat{\varphi}_{\by\longrightarrow \by'}$  with respect to $\bbeta$. Thus we  compute for $k = 1, \ldots, \vert E \vert$,
 \begin{equation} \label{diff last E}
     \frac{\partial \hat{\varphi}_{\by \rightarrow \by'}}{\partial \beta_k} = 
     \begin{cases}
    \frac{1}{{\bx}^{\by}}, & \text{ if } {k = {\by} \rightarrow {\by}'}, \\
    0, & \text{ otherwise.}
    \end{cases}
\end{equation}
Using the assumption $\mathrm{d} \hat{\varphi} (\bv)= \mathbf{0}$, \eqref{eq:v=tvtb}, \eqref{diff first n} and \eqref{diff last E}, we get that
\begin{equation} \label{dv = 0}
\begin{split}
\frac{\partial \hat{\varphi}_{\by\rightarrow \by'}} {\partial (\bx, \bm\beta)} \cdot \bv
& = \sum^{n}_{j=1} \frac{\partial \hat{\varphi}_{\by\rightarrow \by'}}{\partial x_j} v_j + \sum^{\vert E \vert}_{k=1} \frac{\partial \hat{\varphi}_{\by\rightarrow \by'}}{\partial \beta_{k}} v_{n+k}
\\& = \frac{\beta_{\by\rightarrow \by'}}{\bx^{\by}}\frac{-y_1}{x_1} v_1 + \ldots + \frac{\beta_{\by\rightarrow \by'}}{\bx^{\by}}\frac{-y_n}{x_n} v_n
+ \frac{1}{\bx^{\by}} \widetilde{\beta}_{\by\rightarrow \by'}
= 0.
\end{split}
\end{equation}
Multiplying ${\bx}^{\by} > 0$ on both sides of \eqref{dv = 0}, we obtain for any $\by \rightarrow \by'\in E$
\begin{equation} \label{eq:betaBeta}
     \beta_{\by \rightarrow \by'}  \big ( \sum_{i=1}^n \frac{y_i}{x_i} v_i \big ) = \widetilde{\beta}_{\by \rightarrow \by'}.
\end{equation}

Here we introduce the weight functions $c (\by)$ as follows: 
\begin{equation} \label{def: c}
c(\by) := \sum_{i=1}^n \frac{y_i}{x_i} v_i \ \text{ for any } \by \in V.
\end{equation}
Applying \eqref{eq:betaBeta} on all reactions in $G = (V, E)$, we have 
\begin{equation} \label{eq:betaBeta 1}
c(\by) \beta_{\by \rightarrow \by'}
= \widetilde{\beta}_{\by \rightarrow \by'} \ \text{ for any } \by \rightarrow \by'\in E.
\end{equation}
By \eqref{eq:v=tvtb}, we have $\widetilde{\bbeta} = (\widetilde{\beta}_1, \ldots, \widetilde{\beta}_{\vert E\vert}) \in \widetilde{\bB}(G)$.
This shows that, under the weight function $c(\by)$, the vector $\bbeta \in \bB (G)$ is mapped to $\tilde{\bB} (G)$.

Suppose $G = (V, E)$ has $\ell \geq 1$ connected components $\{ V_1, \cdots, V_{\ell} \}$.
By Lemma \ref{lem:constant_weight_flux}, we obtain that for any $1 \leq i \leq \ell$, the weight function $c(\cdot)$ is constant within each connected component $V_i \subset V$.
Hence there is a set of real numbers $\{ k_i \}^{l}_{i=1} \in \mathbb{R}$, such that
\begin{equation} \label{c constant Li}
c(\by) = k_i 
\ \text{ for any } \by \in V_i \subset V.
\end{equation}
From \eqref{def: c} and \eqref{c constant Li}, for any $\by  = (y_1, \cdots, y_n) \in V_i$ we get
\begin{equation} \label{c constant Li_2}
\sum_{j=1}^n \frac{y_j}{x_j} v_j = k_i.
\end{equation}
Then we pick one vertex in each connected component, denoted by $\by^i = (y^i_{1}, \cdots, y^i_{n}) \in V_i$ with $1 \leq i \leq \ell$.
From \eqref{c constant Li_2}, for any $1 \leq i \leq \ell$ and $\by \in V_i$,
\begin{equation} \label{delta c = 0 Li}
c(\by) - c(\by^i)
= \sum_{j=1}^n \frac{y_j}{x_j} v_j - \sum_{j=1}^n \frac{y^i_{j}}{x_j} v_j 
= \sum_{j=1}^n \frac{(y_j - y^i_{j})}{x_j} v_j = 0.
\end{equation}
We also note that $\mathcal{S} = \spn \{ \by - \by^i : \by \in V_i 
\ \text{with} \ 
1 \leq i \leq \ell \}$. 
By \eqref{eq:v=tvtb}, we have $\tilde{\bv} = (v_1, \ldots,v_n) \in \mathcal{S}$ and thus $\tilde{\bv}$ can be written as
\begin{equation} \label{tilde v Li}
\tilde{\bv} = \sum\limits^{\ell}_{i=1} \sum\limits_{\by \in V_i} w_{\by} (\by - \by^i)
\ \text{ with } \
w_{\by} \in \mathbb{R}.
\end{equation}
For every vertex $\by \in V$, we multiply $c(\by) - c(\by^i)$ in \eqref{delta c = 0 Li} by $w_{\by}$, and then do the summation of all such multiplications. Hence we derive that
\begin{equation} 
\label{weight d c = 0 Li 0}
0 = \sum\limits^{\ell}_{i=1} \sum\limits_{\by \in V_i} w_{\by} \big( c(\by) - c(\by^i) \big)
= \sum\limits^{\ell}_{i=1} \sum\limits_{\by \in V_i} w_{\by} \bigg( \sum_{j=1}^n \frac{(y_j - y^i_{j})}{x_j} v_j \bigg).
\end{equation}
Inputting \eqref{tilde v Li} into \eqref{weight d c = 0 Li 0}, we have
\begin{equation} 
\label{weight d c = 0 Li}
\begin{split}
0 & = \sum\limits^{\ell}_{i=1} \sum\limits_{\by \in V_i} w_{\by} \bigg( \sum_{j=1}^n \frac{(y_j - y^i_{j})}{x_j} \Big( \sum\limits^{\ell}_{i=1} \sum\limits_{\by \in V_i} w_{\by} (y_{j} - y^i_{j}) \Big) \bigg)
\\& = \sum\limits^{\ell}_{i=1} \sum\limits_{\by \in V_i} \bigg( \sum_{j=1}^n \frac{w_{\by} (y_j - y^i_{j})}{x_j} \Big( \sum\limits^{\ell}_{i=1} \sum\limits_{\by \in V_i} w_{\by} (y_{j} - y^i_{j}) \Big) \bigg)
\\& = \sum_{j=1}^n
\Bigg(
\sum\limits^{\ell}_{i=1} \sum\limits_{\by \in V_i} \frac{w_{\by} (y_{j} - y^i_{j})}{x_j} \Big( \sum\limits^{\ell}_{i=1} \sum\limits_{\by \in V_i} w_{\by} (y_{j} - y^i_{j}) \Big) \Bigg)
\\& = \sum_{j=1}^n \frac{\Big( \sum\limits^{\ell}_{i=1} \sum\limits_{\by \in V_i} w_{\by} (y_{j} - y^i_{j}) \Big)^2 }{x_j}.
\end{split}
\end{equation}
Since $\bx = (x_{1}, \cdots, x_{n}) \in \mathbb{R}_{>0}^n$, \eqref{weight d c = 0 Li} implies that
\begin{equation}
v_j = \sum\limits^{\ell}_{i=1} \sum\limits_{\by \in V_i} w_{\by} (y_{j} - y^i_{j}) = 0
\ \text{ for } \ 
j = 1, \ldots, n.
\end{equation}
This shows $\tilde{\bv} = (v_1, \ldots,v_n) = \mathbf{0}$.
Then we apply $\tilde{\bv} = \mathbf{0}$ on \eqref{eq:betaBeta 1}, and obtain 
\begin{equation}
\widetilde{\beta}_{\by \rightarrow \by' } = c(\by) \beta_{\by \rightarrow \by'}
= \big( \sum_{i=1}^n \frac{y_i}{x_i} v_i \big) \beta_{\by \rightarrow \by'}
= 0
\ \text{ for any } \
\by \rightarrow \by' \in E,
\end{equation}
which indicates $\widetilde{\bbeta} = (\widetilde{\beta}_1, \ldots, \widetilde{\beta}_{\vert E\vert}) = \mathbf{0}$.
However, this contradicts with $\bv = (\tilde{\bv}, \widetilde{\bbeta}) \neq \mathbf{0}$ and we prove the claim.
\end{proof}

\begin{proposition} \label{prop:PhiIsEmbedding}
The map $\varphi: \mathcal{S}_{\bx_0} \times \bB (G) \rightarrow \RR^{|E|}$ from Theorem \ref{thm:mapTildePhi} is a smooth embedding.
\end{proposition}

\begin{proof}
By Definition \ref{def:embedding}, we need to prove that $\varphi$ is a smooth immersion and a topological embedding. 
Theorem \ref{thm:mapTildePhi} shows the map $\varphi$ is a topological embedding, as well as a homeomorphism onto its image, $\mathcal{V}(G)$.
Finally, by Proposition \ref{lemma:smoothImmersion}, we get that $\varphi$ is a smooth immersion.
\end{proof}

Now we are ready to prove the most important result of our paper:
\begin{proof}[Proof of Theorem \ref{thm:submanifold}]
The proof follows directly from Definition \ref{def:embedding} and Proposition \ref{prop:PhiIsEmbedding}.
\end{proof}

\begin{corollary}\label{cor:diffeomorphicToTheProduct}
    The toric locus is diffeomorphic to the product space between the affine invariant polyhedron and the set of complex-balanced flux vectors.
\end{corollary}

\begin{proof}
    This follows directly from Theorem \ref{thm:submanifold} and \cite[Proposition 5.2]{leeSmooth} (Images of Embeddings as Submanifolds).
\end{proof}

\section{Complex-balanced equilibria depend smoothly on the parameters}
\label{sec:smooth_dependent_toric_locus}

In this section, we prove Theorem  \ref{thm:smoothDependanceEquil}, i.e., we show that the complex-balanced equilibria vary smoothly with the parameters (also called reaction rate constants) $\bk$ in the toric locus.  

Let us introduce the map $\hat{Q}$ in Definition \ref{def:mapStep1}. We will show that $\hat{Q}$ is a smooth map.

\begin{definition} \label{def:mapStep1}

Let $G=(V, E)$ be a weakly reversible Euclidean embedded graph and let $\mathcal{S}$ be its associated stoichiometric subspace.
Define the following map:
\begin{equation}
\hat{Q}: \mathcal{V}(G) \rightarrow \mathbb{R}^n,
\end{equation}
such that 
\begin{equation} \label{def:hat_Q}
\hat{Q}(\bk) := \bX^*,
\end{equation}
where $\bX^* \in \mathcal{S}$ and $\exp  (\bX^*)$ is a complex-balanced equilibrium.
\end{definition}

We first show that the map $\hat{Q}$ in Definition \ref{def:mapStep1} is well-defined and smooth. 
To this end, we introduce the Kirchoff matrix and Proposition \ref{prop:iff} below to explain the connection between complex-balanced equilibria and the reaction rates.

\smallskip

Consider a mass-action system $(G, \bk)$. 
Given a vertex $\by_i \in V$, suppose $\by_i$ belongs to a connected component $V_1$ with $|V_1| = m_1$. We construct the
$m_1 \times m_1$ \textbf{Kirchoff matrix} $A_{\bk}$ as follows:
\begin{equation} \label{def: Ak}
[A_{\bk}]_{ji} :=
\begin{cases}
k_{\byi \rightarrow \byj}, & \ \text{if } \  \by_i, \by_j \in V_1 \ \text{and } \byi \rightarrow \byj \in E \\[5pt]
-\sum\limits_{\byi \rightarrow \byj \in E} k_{\byi \rightarrow \byj}, & \ \text{if } \ i =j, \\[5pt]
0, & \ \text{otherwise}.
\end{cases}
\end{equation}
Denote by $K_i$ the minor of the entry in the $i$-th row and the $i$-th column of $A_{{\bk}}$. 
The following proposition (see \cite[Section 2]{cdss2009} and \cite[Proposition 3.9]{Connected}) gives a characterization of the complex-balanced equilibria.

\begin{proposition}[{\cite[Section 2]{cdss2009},\cite[Proposition 3.9]{Connected}}]
\label{prop:iff}
Let $(G,{\bk})$ be a weakly reversible mass-action system, with $\ell$ connected components. 
For any two vertices $\byi$ and $\byj$ in $G$, consider the equation:
\begin{equation} \label{eq:binom}
K_i \bx^{\byj}-K_j \bx^{\byi}=0.
\end{equation}
Then the following are equivalent:

(i) Equations (\ref{eq:binom}) are satisfied for every pair of vertices belonging to the same connected component of $G$.

(ii) $\bx$ is a complex-balanced equilibrium for the reaction rate vector ${\bk}$.
\end{proposition}

The next Lemma \ref{Q hat smooth} shows the smoothness of the map $\hat{Q}$ defined in \eqref{def:hat_Q}. A similar conclusion was  obtained in \cite[Corollary 3.13]{Connected}.

\begin{lemma}[\cite{Connected}] \label{Q hat smooth}
The map $\hat{Q}$ from Definition \ref{def:mapStep1} is well-defined and smooth.
\end{lemma}

\begin{proof}
Note that, for the completeness of this paper, we explain here the proof of Lemma \ref{Q hat smooth}, which was given in the proof of \cite[Theorem 3.5]{Connected}. 

 Suppose $\bx^* \in \RR^n_{>0}$ is a complex-balanced steady state of $(G, \bk)$. Then by Theorem \ref{thm:HJ} (b), every complex-balanced steady state $\bx \in \RR^n_{>0}$ of the complex-balanced system $(G, \bk)$ satisfies the following
\[
\ln (\bx)  \in \ln \bx^* + \mathcal{S}^\perp.
\]
Hence, in particular, for any $\bk \in \mathcal{V}(G)$, there exists a unique complex-balanced steady state $\bx$, such that
\[
\ln (\bx) \in \mathcal{S}.
\] 
Therefore, the map $\hat{Q}$ is well-defined.

\medskip

Next, we show the smoothness of the map $\hat{Q}$. 
From Definition \ref{def: Ak}, it is clear that the vector ${\bK} = (K_i)  \in \mathbb{R}_{>0}^m$ depends smoothly on the reaction rate vector ${\bk} = (k_{ij}) \in \mathbb{R}_{>0}^{E}$. Hence, it suffices to show that $\hat{Q}(\bk) = \bX^*$ depends  smoothly on ${\bK}$.

By Proposition \ref{prop:iff}, a state $\bx$ is a complex-balanced equilibrium if and only if for any two vertices $\byi, \byj$ in the same connected component of $G$ we have
\begin{equation} \label{eq:binom2}
K_i \bx^{\byj}
= K_j \bx^{\byi}.
\end{equation}
Taking the log of both sides in Equation \eqref{eq:binom2}, we derive
\begin{equation} \label{eq:logMultivar1}
\ln (K_i) + \byj^\intercal \cdot \ln (\bx)
=\ln (K_j) + \byi^\intercal \cdot \ln (\bx).
\end{equation} 
By setting $\bX = \ln (\bx)$, we can rewrite (\ref{eq:logMultivar1}) as
\begin{equation}\label{eq:logMultivar2}
\ln (K_i / K_j)
= (\byi^\intercal - \byj^\intercal) \cdot \bX,
\end{equation}
where $\byi$ and $\byj$ are two vertices belonging to the same connected component of $G$.

The proof consists of two steps, as it was shown in the proof of \cite[Theorem 3.5]{Connected} (see also \cite[Corollary 3.13]{Connected}).  

First, we prove the smoothness in the case where the graph $G$ has only one connected component.
Second, we generalize the result for any number of connected components.

\paragraph{Step 1.}
Suppose that the graph $G$ has a single connected component. In this case, the vertices $\{\by_1, \ldots, {\boldsymbol{y}_m}\}$ belong to the same connected component of $G$. Note that Equations (\ref{eq:logMultivar2}) are equivalent to the following system of linear equations in the unknowns $\bX$:
\begin{equation} \label{system 1}
\begin{bmatrix}
\ln ( K_1 / K_2) \\
\ln ( K_2 / K_3) \\
\vdots \\
\ln ( K_{m-1} / K_m)
\end{bmatrix}
=
\underbrace{
\begin{bmatrix}
 \by_1^\intercal - \by_2^\intercal \\
\by_2^\intercal - \by_3^\intercal  \\
\vdots \\
\by_{m-1}^\intercal - \by_{m}^\intercal
\end{bmatrix}
}_{\Delta \boldsymbol{y}}
\begin{bmatrix}
 X_1\\
 X_2\\
\vdots \\
X_n
\end{bmatrix}.
\end{equation}
The graph $G$ is strongly connected, hence its stoichiometric subspace (recall Definition \ref{def:stoichiom}) is given by
\begin{equation} \notag
\mathcal{S} = \spn \{ \by_1^\intercal -\by_2^\intercal, \by_2^\intercal -\by_3^\intercal, \ldots, \by_{m-1}^\intercal -\by_{m}^\intercal \}.
\end{equation}
If we denote the dimension of $\mathcal{S}$ by $s$, then we obtain $s \leq \min \{ m-1, n\}$. In addition,  the matrix $\Delta \boldsymbol{y}$ has exactly $s$ linearly independent rows. 
We may assume without loss of generality that the first $s$ rows of $\Delta \boldsymbol{y}$ are linearly independent. Hence, we have:
\begin{equation} \label{S span}
\mathcal{S} = \spn \{ \by_1^\intercal -\by_2^\intercal, \by_2^\intercal -\by_3^\intercal, \ldots, \by_{s}^\intercal -\by_{s+1}^\intercal \}.
\end{equation}

Let us use the following notations:
\begin{equation} \notag
\Delta_s \boldsymbol{y} := 
\begin{bmatrix}
 \by_1^\intercal - \by_2^\intercal \\
\by_2^\intercal - \by_3^\intercal  \\
\vdots \\
\by_s^\intercal - \by_{s+1}^\intercal
\end{bmatrix}, \ \text{and }
\Delta_s {\bK}:=\begin{bmatrix}
K_1 / K_2 \\
K_2 / K_3 \\
\vdots \\
K_{s} / K_{s+1}
\end{bmatrix}.
\end{equation}
Consider now the system of equations (\ref{system 2}):
\begin{equation} \label{system 2}
\ln (\Delta_s {\bK})=(\Delta_s \boldsymbol{y}) \bX.
\end{equation}

By Theorem \ref{thm:HJ}, for any $\bk \in \mathcal{V}(G)$, there exists a complex-balanced steady state $\bx^* \in \RR^n_{>0}$ of the complex-balanced system $(G, \bk)$. In addition, we can write the solutions to \eqref{system 1} as follows:  
$\bX = \ln \bx^* + \mathcal{S}^\perp$ and we obtain that the dimension of the set of solutions to (\ref{system 1}) is $n-s$. 
Moreover, one can check that the solutions of the system (\ref{system 1}) are also solutions of the system (\ref{system 2}) and the set of solutions to (\ref{system 2}) is also of dimension $n-s$. In conclusion, when solving for $\bX$, the systems (\ref{system 1}) and (\ref{system 2}) are equivalent.

 In the sequel, let us construct $\bX^*$, the special solution to the system \eqref{system 2}, where $\bX^* \in \mathcal{S}$. 
Remember that the dimension  of the stoichiometric subspace $s \leq \min \{ m-1, n\}$. We consider the following two cases:

\medskip

\textbf{Case 1: } $s = n$. In this case, we have that the stoichiometric subspace $\mathcal{S} = \mathbb{R}^n$ and the matrix $\Delta_s \by$ is invertible. 
We compute a solution of (\ref{system 2}) as follows:
\begin{equation} \label{X* 1}
\bX^* = (\Delta_s \boldsymbol{y})^{-1} \ln (\Delta_s {\bK}). 
\end{equation}
Then $\bX^* \in \mathcal{S} = \mathbb{R}^n$, and $\exp  (\bX^*)$ is a complex-balanced equilibrium.

\medskip

\textbf{Case 2: } $s < n$.
Remember that we denote the orthogonal complement of $\mathcal{S}$ by $\mathcal{S}^{\perp}$.  
Then we have: 
\begin{equation} \label{S perp}
0 < \dim (\mathcal{S}^{\perp}) = n - \dim (\mathcal{S}) = n - s,
\end{equation}
since the stoichiometric subspace $\mathcal{S} \subset \mathbb{R}^n$.
Let us consider $B = \{ \bv_1, \bv_2, \ldots, \bv_{n-s} \} \subset \mathbb{R}^n$, a basis of $\mathcal{S}^{\perp}$.
Next, let us construct a new matrix and vector as follows:
\begin{equation} \label{tilde matirx}
\tilde{\Delta} \boldsymbol{y} := 
\begin{bmatrix}
\Delta_s \boldsymbol{y} \\
\bv_1^\intercal \\
\vdots \\
\bv_{n-s}^\intercal
\end{bmatrix}, \ \text{and }
\tilde{\Delta} {\bK} := \begin{bmatrix}
\Delta_s {\bK} \\
0 \\
\vdots \\
0
\end{bmatrix}.
\end{equation}
We now focus on system (\ref{system 3}):
\begin{equation} \label{system 3}
\ln (\tilde{\Delta} {\bK}) = (\tilde{\Delta} \boldsymbol{y}) \bX.
\end{equation} 
From \eqref{S span} and the fact that $\{ \bv_1, \ldots, \bv_{n-s} \}$ forms a basis of $\mathcal{S}^{\perp}$, we obtain that $\tilde{\Delta} \boldsymbol{y} \in \mathbb{R}_{n \times n}$ is an invertible matrix.
Therefore, we get a solution of system (\ref{system 3}) as
\begin{equation} \label{X* 2}
	\bX^* = (\tilde{\Delta} \boldsymbol{y})^{-1} \ln (\tilde{\Delta} {\bK}).
\end{equation}
We have that $\bX^*$ solves (\ref{system 2}) and for $i = 1, \cdots, n-s$, 
\begin{equation} \notag
\bv_i^\intercal \cdot \bX^*  = 0.
\end{equation}
Therefore $\bX^* \in \mathcal{S}$. Moreover, by construction, $\exp  (\bX^*)$ is a complex-balanced equilibrium. 

To sum up, in both cases we computed the vector $\bX^* \in \mathcal{S}$, such that $\exp  (\bX^*)$ is a complex-balanced equilibrium of the system $(G, \bk)$.
Since the map $\hat{Q}$ is well-defined, we have that $\hat{Q}(\bk) = \bX^*$.
In addition, since $(\Delta_s \boldsymbol{y})^{-1}$ and $(\tilde{\Delta} \boldsymbol{y})^{-1}$ are fixed real matrices, we conclude that $\bX^*$ depends smoothly on the vector ${\bK}$.

\paragraph{Step 2.}
Let us suppose that the graph $G$ has multiple connected components, denoted by $V_1, \ldots, V_\ell$ with $\ell > 1$.
By relabeling the vertices according to the connected components of $G$ we have:
\[
V_p = \{ \by_{m_{p-1}+1}, \ldots, \by_{m_p}\}, \text{ for } 1 \leq p \leq \ell.
\] 
The system \eqref{eq:logMultivar2} is equivalent to the following system of equations in the unknown $\bX$:
\begin{equation} \label{system 1 l>1}
\underbrace{
\begin{bmatrix}
\ln ( K_1 / K_2) \\
\vdots \\
\ln ( K_{m_1 - 1} / K_{m_1}) 
\\
\hdashline[1.5pt/1.5pt]
\ln ( K_{m_1 + 1} / K_{m_1 + 2}) \\
\vdots \\
\ln ( K_{m_2 - 1} / K_{m_2}) 
\\
\hdashline[1.5pt/1.5pt]
\vdots \\
\ln ( K_{m_{\ell}-1} / K_{m_{\ell}})
\end{bmatrix} 
}_{\ln (\Delta {\bK})}
= 
\underbrace{
\begin{bmatrix}
 \by_1^\intercal - \by_2^\intercal \\
\vdots \\
\by_{m_1 - 1}^\intercal - \by_{m_1}^\intercal 
\\
\hdashline[1.5pt/1.5pt]
\by_{m_1 + 1}^\intercal - \by_{m_1 + 2}^\intercal 
\\
\vdots \\
\by_{m_2 - 1}^\intercal - \by_{m_2}^\intercal  
\\
\hdashline[1.5pt/1.5pt]
\vdots \\
\by_{m_{\ell}-1}^\intercal - \by_{m_{\ell}}^\intercal
\end{bmatrix}
}_{\Delta \boldsymbol{y}}
\begin{bmatrix}
X_1 \\
X_2 \\
\vdots \\
X_n
\end{bmatrix}.
\end{equation}

By Definition \ref{def:stoichiom}, the stoichiometric subspace of $G$ is 
\begin{equation} \notag
\mathcal{S} = \spn \{ \by_1^\intercal -\by_2^\intercal, \ldots, \by_{m_1 - 1}^\intercal - \by_{m_1}^\intercal, 
\by_{m_1 + 1}^\intercal - \by_{m_1 + 2}^\intercal, \ldots, \by_{m_{\ell}-1}^\intercal -\by_{m_{\ell}}^\intercal \}.
\end{equation}
As before, letting $s$ denote the dimension of  $\mathcal{S}$, we have $s \leq \min \{ m-\ell, n\}$  and the matrix $\Delta \boldsymbol{y}$ has exactly $s$ linearly independent rows.
Then we choose $s$ linearly independent rows in $\Delta \boldsymbol{y}$, and these rows belonging to  
$\Delta \boldsymbol{y}$ provide a matrix $\Delta_s \boldsymbol{y}$ which is full row rank. The vector $\ln (\Delta_s {\bK})$ is given by the corresponding rows in $\ln (\Delta {\bK})$. 
The system (\ref{system 1 l>1}) is equivalent to the following one in the unknowns $\bX$:
\begin{equation} \label{system 2 l>1}
\ln (\Delta_s {\bK})=(\Delta_s \boldsymbol{y}) \bX.
\end{equation}

Subsequently, let us construct $\bX^* \in \mathcal{S}$, the special solution where $\exp  (\bX^*)$ is a complex-balanced equilibrium.
As we did in the first step we consider two cases, in function of $s$.

First, if $s = n$, then the matrix $\Delta_s \boldsymbol{y}$ is invertible.
Hence, we compute a solution of system (\ref{system 2 l>1}) as follows: 
\begin{equation} \notag
\bX^* = (\Delta_s \boldsymbol{y})^{-1} \ln (\Delta_s {\bK})  \in \mathcal{S} = \mathbb{R}^n,
\end{equation}
and $\exp  (\bX^*)$ is a complex-balanced equilibrium.

Second, if $s < n$, let $B = \{ \bv_1, \ldots, \bv_{n-s} \}$ be a basis of $\mathcal{S}^{\perp}$. 
Analogously to Equations \eqref{tilde matirx}-\eqref{X* 2}, we start by adding $\bv_1^\intercal, \ldots, \bv_{n-s}^\intercal$ on the bottom of the matrix $\Delta_s \boldsymbol{y}$, and add $n-s$ zeros to the vector $\ln (\Delta_s {\bK})$.
Next, we get the desired solution $\bX^* \in \mathcal{S}$ of \eqref{system 2 l>1}, and $\exp  (\bX^*)$ is a complex-balanced equilibrium.

As in the case of Step 1 (the single connected component case), we conclude that $\bX^* = \hat{Q}(\bk)$ depends smoothly on the vector ${\bK}$.
\end{proof}

\section{Complex-balanced equilibria depend smoothly on the initial conditions}
\label{sec:smooth_dependent_initial}


Here our goal is to prove Theorem \ref{thm:smoothDependanceinitial}, i.e., to show that the complex-balanced equilibria vary smoothly with the initial condition vector $\bx_0$. 

This section is divided into the two logical steps of the proof. First, we introduce the map $\phi$ in Definition \ref{def:mapStep2} and we show that $\phi$ is a smooth map in Lemma \ref{phi smmoth}. Second, we give another map $\Phi$ in Definition \ref{def:mapStep3}. Lemma \ref{Phi smmoth} is dedicated to proving that $\Phi$ is also smooth.

\begin{definition} \label{def:mapStep2}
Let $G=(V, E)$ be a weakly reversible Euclidean embedded graph and let $\mathcal{S}$ be its associated stoichiometric subspace.
Given a fixed state $\bx_0 \in \mathbb{R}_{>0}^n$, define the following map: 
\begin{equation}
\phi: \mathbb{R}^n \rightarrow \bx_0 + \mathcal{S},
\end{equation}
such that
\begin{equation}
\phi (\bX^*) := \bx^*,
\end{equation} 
where $\bx^* = \exp (\bX^* + \mathcal{S}^\perp) \cap ({\boldsymbol{x}_0} + \mathcal{S})$.
\end{definition}

\begin{lemma} \label{phi smmoth}
The map $\phi$ from Definition \ref{def:mapStep2} is well-defined and smooth.
\end{lemma}

\begin{proof}

Using the well-known Birch Theorem (see, for instance \cite[Proposition 10]{cdss2009}, \cite[Theorem 1.10]{pachterStumrfels}), for any $\bX^* \in \mathbb{R}^n$ and $\bx_0 \in \mathbb{R}_{>0}^n$, then
$(\bx_0 + \mathcal{S})$ and $\exp (\bX^* + \mathcal{S}^{\perp})$ have a unique intersection point. Therefore, the map $\phi$ is well-defined. 

\medskip

Next, we show the smoothness of the map $\phi$. Let $s$ be the dimension of  $\mathcal{S}$, we have 
\[
\dim (\mathcal{S}) = s \leq n, 
\ \text{ and } \
\dim (\mathcal{S}^{\perp})  = n - \dim (\mathcal{S}) = n - s.
\]
Then we consider $s$ in two cases: $s = n$ and $s < n$.

\medskip

\textbf{Case 1: } $s = n$.
This shows $\mathcal{S} = \mathbb{R}^n$ and $\mathcal{S}^{\perp} = \emptyset$. 
Thus we derive $\bx_0 + \mathcal{S} = \mathbb{R}^n$ and
\begin{equation}
\phi (\bX^*) = \exp (\bX^*+\mathcal{S}^\perp ) \cap (\bx_0+\mathcal{S}) = \exp(\bX^*).
\end{equation}
It is clear that the map $\phi$ is smooth in this case.

\medskip

\textbf{Case 2: } $s < n$.
Then $\mathcal{S}^{\perp} \neq \emptyset$ and $\dim (\mathcal{S}^{\perp})  = n - s > 0$.
Then we pick a basis of $\mathcal{S}^{\perp}$, denoted by $B = \{ \bv_1, \bv_2, \ldots, \bv_{n-s} \} \subset \mathbb{R}^n$.
Given $\bx_0 \in \mathbb{R}_{>0}^n$, the Birch Theorem shows that for any $\bX^* \in \mathbb{R}^n$, there exists a unique real vector $\bw = ( w_1, \ldots, w_{n-s} )$, such that
\begin{equation} \label{eq:w}
\phi (\bX^*) = \exp (\bX^* +
\sum^{n-s}_{i=1} w_i \bv_i ) \in \bx_0 + \mathcal{S}.
\end{equation}
Thus each $w_i$ can be considered as a real function of $\bX^*$, i.e., $w_i = w_i (\bX^*) \in \mathbb{R}$ for $1 \leq i \leq n-s$. Since $\{ \bv_1, \bv_2, \ldots, \bv_{n-s} \}$ forms a basis of $\mathcal{S}^{\perp}$, \eqref{eq:w} is equivalent to the following:
for any $\bX^* \in \mathbb{R}^n$, the real vector function $( w_1 (\bX^*), \ldots, w_{n-s} (\bX^*))$ satisfies

\begin{equation} \label{eq:w 2}
\exp (\bX^* + \sum^{n-s}_{i=1} w_i (\bX^*) \bv_i ) \cdot \bv_i = \bx_0 \cdot \bv_i,
\ \text{for} \
1 \leq i \leq n-s.
\end{equation}

Now we construct $n-s$ functions $\{ f_1, \ldots, f_{n-s} \}$ as follows:
given $\bx_0 \in \mathbb{R}_{>0}^n$ and suppose $\{ \bv_1, \bv_2, \ldots, \bv_{n-s} \} \subset \mathbb{R}^n$ is a basis of $\mathcal{S}^{\perp}$, then for $1 \leq i \leq n-s$
\begin{equation}
f_i: \mathbb{R}^n \times \mathbb{R}^{n-s} \rightarrow \mathbb{R},
\end{equation}
such that for $(\bX, \by) \in \mathbb{R}^n \times \mathbb{R}^{n-s}$,
\begin{equation}
f_i (\bX, \by) :=
\exp (\bX +
\sum^{n-s}_{i=1} y_i \bv_i ) \cdot \bv_i
- \bx_0 \cdot \bv_i.
\end{equation}
It is clear that the solution to \eqref{eq:w 2} is equivalent to say that for any $\bX^* \in \RR^n$,
\[
f_1 = \cdots = f_{n-s} = 0,
\ \text{at} \ 
(\bX, \by) = (\bX^*, w_1 (\bX^*), \cdots, w_{n-s} (\bX^*) ).
\]

Here we claim that for any $\bX^* \in \RR^n$, we have
\begin{equation} 
\label{det non-zero}
\det \Big( \frac{\p f_i}{\p y_j} \Big) \neq 0,
\ \text{at} \ 
(\bX, \by) = (\bX^*, w_1 (\bX^*), \cdots, w_{n-s} (\bX^*) ).
\end{equation}
By direct computation, we get for $1 \leq i, j \leq n-s$,
\begin{equation} \notag
\bJ_{\mathbf{f}} := 
\Big( \frac{\p f_i}{\p y_j} \Big)_{i, j} = 
\exp (\bX +
\sum^{n-s}_{i=1} y_i \bv_i ) \cdot ( \bv_i \bv_j ),
\end{equation}
where $\bv_i \bv_j := (v_{i, 1} v_{j, 1}, v_{i, 2} v_{j, 2}, \cdots, v_{i, n} v_{j, n})$.
We can rewrite the matrix $\bJ_{\mathbf{f}}$ as
\begin{equation} \label{Jf}
\bJ_{\mathbf{f}} = 
\begin{bmatrix}
\bv_1 \\
\vdots \\
\bv_{n-s}
\end{bmatrix}
\begin{bmatrix}
    \tilde{x}_{1} & & \\
    & \ddots & \\
    & & \tilde{x}_{n}
\end{bmatrix}
\begin{bmatrix}
\bv^{\intercal}_1, \cdots, \bv^{\intercal}_{n-s} \\
\end{bmatrix},
\end{equation} 
where $(\tilde{x}_{1}, \cdots, \tilde{x}_{n}) := \exp (\bX +
\sum\limits^{n-s}_{i=1} y_i \bv_i )$.
Supposing $\det (\bJ_{\mathbf{f}}) = 0$ at a point $(\hat{\bX}, \hat{\by})$, then there exists a non-zero column vector $\blam \in \mathbb{R}^{n-s}$, such that 
\begin{equation} \notag
\mathbf{J}_{\mathbf{f}} (\hat{\bX}, \hat{\by}) \cdot \blam = \mathbf{0}.
\end{equation}
This implies that
\begin{equation} \label{lambda J}
\blam^{\intercal} \cdot \mathbf{J}_{\mathbf{f}}  (\hat{\bX}, \hat{\by}) \cdot \blam = 0.
\end{equation}
From \eqref{Jf}, \eqref{lambda J} and $\exp (\bX +
\sum\limits^{n-s}_{i=1} y_i \bv_i ) \in \RR^n_{>0}$, we obtain
\begin{equation} \notag
\begin{bmatrix}
\bv^{\intercal}_1, \cdots, \bv^{\intercal}_{n-s} \\
\end{bmatrix} 
\cdot \blam = \mathbf{0}.
\end{equation}
Since $\{ \bv_1, \cdots, \bv_{n-s} \}$ forms a basis for $\mathcal{S}^{\perp}$ and they are linearly independent, we get  $\blam = \mathbf{0}$ and this contradicts  the assumption. Thus, we proved the claim.

Note that $\{ f_1, \ldots, f_{n-s} \}$ are all smooth functions. Using \eqref{det non-zero} and the Implicit Function Theorem, we obtain that there exists a unique smooth function 
\[
g: \mathbb{R}^n \rightarrow \mathbb{R}^{n-s}
\]
such that for any $\bX \in \mathbb{R}^n$,
\begin{equation}
f_i (\bX, g(\bX)) = 0, 
\ \text{for} \ i = 1, \cdots, n-s.
\end{equation}
Using the Birch Theorem, we deduce 
that given a state $\bx_0 \in \mathbb{R}^n_{>0}$, for any $\bX^* \in \RR^n$
\begin{equation}
g(\bX^*) = (w_1 (\bX^*), \ldots, w_{n-s} (\bX^*)).
\end{equation}
This shows that $w_1 (\bX^*), \ldots, w_{n-s} (\bX^*)$ are smooth functions with respect to $\bX^*$.
Therefore, we conclude that the map
$\phi (\bX^*) = \exp (\bX^* +
\sum^{n-s}_{i=1} w_i \bv_i )$ is smooth.
\end{proof}

Now we are ready to prove Theorem \ref{thm:smoothDependanceEquil}.

\begin{proof}[Proof of Theorem \ref{thm:smoothDependanceEquil}]

Given an initial state ${\boldsymbol{x}_0} \in \mathbb{R}_{>0}^n$,  we define the following map, introduced in \cite[Definition 3.1]{Connected}:
\begin{equation} \label{equ:q}
Q_{{\boldsymbol{x}_0}}: \mathcal{V}(G) \rightarrow ({\boldsymbol{x}_0}+\mathcal{S}) \cap \mathbb{R}_{>0}^n,
\end{equation}
such that for any $\bk \in \mathcal{V}(G)$, $Q_{{\boldsymbol{x}_0}}({\bk})$ is the unique complex-balanced equilibrium in the invariant polyhedron $\mathcal{S}_{\bx_0}$, under the mass-action system $(G, \bk)$.

The map $Q_{{\boldsymbol{x}_0}}$ is well-defined for any state ${\boldsymbol{x}_0} \in \mathbb{R}_{>0}^n$ and $\bk \in \mathcal{V}(G)$, since Theorem \ref{thm:HJ} shows every complex-balanced system admits a unique equilibrium within each invariant polyhedron.  
Using Lemma \ref{Q hat smooth}, we obtain that for any $\bk \in \mathcal{V}(G)$, we have
\[
\hat{Q} (\bk) = \bX^* \in \mathcal{S}, 
\]
and $\exp (\bX^*)$ is a complex-balanced equilibrium of the mass-action system $(G, \bk)$.
On the other hand, from Lemma \ref{phi smmoth}, for any $\bX^* \in \RR^n$ we have
\begin{equation} \notag
\phi (\bX^*) = \exp (\bX^* + \mathcal{S}^\perp) \cap ({\boldsymbol{x}_0} + \mathcal{S}).
\end{equation} 

From Theorem \ref{thm:HJ}, every complex-balanced equilibrium can be written as
\begin{equation}
\exp (\bX^* + \mathcal{S}^\perp ).
\end{equation}
Furthermore, 
$\exp(\bX^* + \mathcal{S}^\perp )\cap (\bx_0+\mathcal{S})$ is the unique positive complex-balanced equilibrium in the invariant polyhedron $\mathcal{S}_{\bx_0}$. 
Thus, given $\bx_0 \in \mathbb{R}_{>0}^{n}$ and for any $\bk \in \mathcal{V}(G)$, we deduce
\begin{equation}
Q_{\bx_0} (\bk) = \phi \circ \hat{Q} (\bk).
\end{equation}
By Lemma \ref{Q hat smooth} and Lemma \ref{phi smmoth}, we get that both $\hat{Q}$ and $\phi$ are smooth functions. Therefore, we conclude the smoothness of the map $Q_{x_0}$ and prove the theorem.
\end{proof}

To prove Theorem \ref{thm:smoothDependanceinitial}, we introduce another map in Definition \ref{def:mapStep3} and show it is smooth.

\begin{definition} \label{def:mapStep3}
Let $(G, \bk)$ be a weakly reversible mass-action system with $\bk \in \mathcal{V}(G)$, and let $\mathcal{S}$ be its associated stoichiometric subspace.
Given fixed $\bX^* \in \mathbb{R}^n$, define the following map: 
\begin{equation}
\Phi: \mathbb{R}_{>0}^n \rightarrow \mathbb{R}_{>0}^n,
\end{equation}
such that
\begin{equation}
\Phi (\bx_0) := \bx^*,
\end{equation} 
where $\bx^* = \exp (\bX^* + \mathcal{S}^\perp) \cap ({\boldsymbol{x}_0} + \mathcal{S})$.
\end{definition}

\begin{lemma} \label{Phi smmoth}
The map $\Phi$ from Definition \ref{def:mapStep3} is well-defined and smooth.
\end{lemma}

\begin{proof}

Similar to the proof of Lemma \ref{phi smmoth}, the Birch Theorem shows that the map $\Phi$ is well-defined. 

\medskip

To show the smoothness of the map $\Phi$, we let $s$ be the dimension of  $\mathcal{S}$, and we consider $s$ in two cases: $s = n$ and $s < n$.

\medskip

\textbf{Case 1: } $s = n$.
This shows $\mathcal{S} = \mathbb{R}^n$ and $\mathcal{S}^{\perp} = \emptyset$. 
Thus we derive
\begin{equation}
\phi (\bX^*) = \exp (\bX^*+\mathcal{S}^\perp ) \cap (\bx_0+\mathcal{S}) = \exp(\bX^*).
\end{equation}
Since $\bX^*$ is a fixed vector, the map $\Phi$ is a constant function and it is smooth in this case.

\medskip

\textbf{Case 2: } $s < n$.
In this case, $\mathcal{S}^{\perp} \neq \emptyset$ and $\dim (\mathcal{S}^{\perp})  = n - s > 0$.
Then we pick a basis of $\mathcal{S}^{\perp}$, denoted by $B = \{ \bv_1, \bv_2, \ldots, \bv_{n-s} \} \subset \mathbb{R}^n$.
Given $\bX^* \in \mathbb{R}^n$, the Birch Theorem shows that for any $\bx_0 \in \mathbb{R}^n_{>0}$, there exists a unique real vector $\bz = ( z_1, \ldots, z_{n-s} )$, such that
\begin{equation} \label{eq:z}
\Phi (\bx_0) = \exp (\bX^* +
\sum^{n-s}_{i=1} z_i \bv_i ) \in \bx_0 + \mathcal{S}.
\end{equation}
Thus each $z_i$ is a real function of $\bx_0$, i.e., $z_i = z_i (\bx_0) \in \mathbb{R}$ for $1 \leq i \leq n-s$. 
Since $\{ \bv_1, \bv_2, \ldots, \bv_{n-s} \}$ forms a basis of $\mathcal{S}^{\perp}$, \eqref{eq:z} is equivalent to the following:
for any $\bx_0 \in \mathbb{R}^n_{>0}$, the real vector function $( z_1 (\bx_0), \ldots, z_{n-s} (\bx_0))$ satisfies
\begin{equation} \label{eq:z 2}
\exp (\bX^* + \sum^{n-s}_{i=1} z_i (\bx_0) \bv_i ) \cdot \bv_i = \bx_0 \cdot \bv_i,
\ \text{for} \
1 \leq i \leq n-s.
\end{equation}

Now we construct $n-s$ functions $\{ f_1, \ldots, f_{n-s} \}$ as follows:
given $\bX^* \in \mathbb{R}^n$ and suppose $\{ \bv_1, \bv_2, \ldots, \bv_{n-s} \} \subset \mathbb{R}^n$ is a basis of $\mathcal{S}^{\perp}$, then for $1 \leq i \leq n-s$
\begin{equation}
f_i: \mathbb{R}^n_{>0} \times \mathbb{R}^{n-s} \rightarrow \mathbb{R},
\end{equation}
such that for $(\bx, \by) \in \mathbb{R}^n \times \mathbb{R}^{n-s}$,
\begin{equation}
f_i (\bx, \by) :=
\exp (\bX^* +
\sum^{n-s}_{i=1} y_i \bv_i ) \cdot \bv_i
- \bx \cdot \bv_i.
\end{equation}
It is clear that the solution to \eqref{eq:z 2} is equivalent to say that for any $\bx_0 \in \mathbb{R}^n_{>0}$,
\[
f_1 = \cdots = f_{n-s} = 0,
\ \text{at} \ 
(\bx, \by) = (\bx_0, z_1 (\bx_0), \cdots, z_{n-s} (\bx_0) ).
\]

By the proof of Lemma \ref{phi smmoth}, we obtain that for any $\bx_0 \in \mathbb{R}^n_{>0}$, 
\begin{equation} 
\label{det non-zero_b}
\det \Big( \frac{\p f_i}{\p y_j} \Big) \neq 0,
\ \text{at} \ 
(\bx, \by) = (\bx_0, z_1 (\bx_0), \cdots, z_{n-s} (\bx_0) ).
\end{equation}
Since $\{ f_1, \ldots, f_{n-s} \}$ are all smooth functions, we apply \eqref{det non-zero_b} on the Implicit Function Theorem. Then we derive that there exists a unique smooth function 
\[
g: \mathbb{R}^n \rightarrow \mathbb{R}^{n-s}
\]
such that for any $\bx_0 \in \mathbb{R}^n_{>0}$,
\begin{equation}
f_i (\bX, g(\bX)) = 0, 
\ \text{for} \ i = 1, \cdots, n-s.
\end{equation}
The Birch Theorem implies that for any $\bx_0 \in \mathbb{R}^n_{>0}$,
\begin{equation}
g(\bx_0) = (z_1 (\bx_0), \ldots, z_{n-s} (\bx_0)).
\end{equation}
This shows that $z_1 (\bx_0), \ldots, z_{n-s} (\bx_0)$ are smooth functions of $\bx_0$.
Therefore, we conclude that 
given $\bX^* \in \mathbb{R}^n$, the map
$\Phi (\bx_0) = \exp (\bX^* +
\sum^{n-s}_{i=1} z_i (\bx_0) \bv_i )$ is smooth.
\end{proof}

Finally, we are ready to prove Theorem \ref{thm:smoothDependanceinitial}.

\begin{proof}[Proof of Theorem \ref{thm:smoothDependanceinitial}]

Given the vector of reaction rate constants $\bk \in \mathcal{V}(G)$, we define the following map:
\begin{equation} \label{equ:p}
P: \mathbb{R}_{>0}^n \rightarrow \mathbb{R}_{>0}^n,
\end{equation}
such that for any $\bx_0 \in \RR^n_{>0}$, we define $P (\bx_0)$ to be the unique complex-balanced equilibrium in the invariant polyhedron $\mathcal{S}_{\bx_0}$, under the mass-action system $(G, \bk)$.

By Theorem \ref{thm:HJ}, the map $P$ is well-defined and $P (\bx_0) = \exp(\bX^* + \mathcal{S}^\perp )\cap (\bx_0+\mathcal{S})$.
 By Lemma \ref{Q hat smooth}, we obtain that for any $\bk \in \mathcal{V}(G)$,
$\hat{Q} (\bk) = \bX^* \in \mathcal{S}$ 
and $\exp (\bX^*)$ is a complex-balanced equilibrium of the mass-action system $(G, \bk)$.
Then Lemma \ref{Phi smmoth} shows that, given $\bX^* \in \RR^n$ and for any $\bx_0 \in \RR^n_{>0}$, we have
\begin{equation} \notag
\Phi (\bx_0) = \exp (\bX^* + \mathcal{S}^\perp) \cap ({\boldsymbol{x}_0} + \mathcal{S}).
\end{equation}  
Therefore, given $\bk \in \mathcal{V}(G)$ and for any $\bx_0 \in \mathbb{R}_{>0}^{n}$, we have
\begin{equation}
P (\bx_0) = \Phi \circ \hat{Q} (\bx_0).
\end{equation}
By Lemma \ref{Q hat smooth} and Lemma \ref{Phi smmoth}, both $\hat{Q}$ and $\Phi$ are smooth functions. Hence we conclude the smoothness of the map $P$ and prove the theorem.
\end{proof}


\section*{Ackowledgements}
G. Craciun and M.-Ş. Sorea are thankful for the support of the Nonlinear Algebra research group of Bernd Sturmfels at the Max
Planck Institute for Mathematics in the Sciences in Leipzig, Germany. M.-Ş. Sorea
is grateful to Antonio Lerario for the very supportive working conditions during her
postdoc at SISSA, in Trieste, Italy. G. Craciun's work was supported in part by the National Science Foundation grant DMS-2051568.

\printbibliography

@book {leeSmooth,
    AUTHOR = {Lee, John M.},
     TITLE = {Introduction to smooth manifolds},
    SERIES = {Graduate Texts in Mathematics},
    VOLUME = {218},
   EDITION = {Second},
 PUBLISHER = {Springer, New York},
      YEAR = {2013},
     PAGES = {xvi+708},
      ISBN = {978-1-4419-9981-8},
   MRCLASS = {58-01 (53-01 57-01)},
  MRNUMBER = {2954043},
}

@article{Connected,
AUTHOR={Craciun, Gheorghe and Jin, Jiaxin and Sorea, Miruna-{\c{S}}tefana},
TITLE={The structure of the moduli space of toric dynamical systems of a reaction network},
URL={https://arxiv.org/abs/2008.11468},
}

@book {shiu10,
    AUTHOR = {Shiu, Anne},
     TITLE = {Algebraic methods for biochemical reaction network theory},
      NOTE = {Thesis (Ph.D.)--University of California, Berkeley},
 PUBLISHER = {ProQuest LLC, Ann Arbor, MI},
      YEAR = {2010},
     PAGES = {116},
}

@article{gharahi2020fine,
  title={Fine-structure classification of multiqubit entanglement by algebraic geometry},
  author={Gharahi, Masoud and Mancini, Stefano and Ottaviani, Giorgio},
  journal={Physical Review Research},
  volume={2},
  number={4},
  pages={043003},
  year={2020},
  publisher={APS}
}

@book {MR4423369,
    AUTHOR = {Micha\l ek, Mateusz and Sturmfels, Bernd},
     TITLE = {Invitation to nonlinear algebra},
    SERIES = {Graduate Studies in Mathematics},
    VOLUME = {211},
 PUBLISHER = {American Mathematical Society, Providence, RI},
      YEAR = {2021},
     PAGES = {xiii+226},
   MRCLASS = {14-02 (13-02 13A50 13P10 14T10 15A69 20G05)},
  MRNUMBER = {4423369},
       DOI = {10.1090/gsm/211},
       URL = {https://doi.org/10.1090/gsm/211},
}

@book {MR1288523,
    AUTHOR = {Griffiths, Phillip and Harris, Joseph},
     TITLE = {Principles of algebraic geometry},
    SERIES = {Wiley Classics Library},
      NOTE = {Reprint of the 1978 original},
 PUBLISHER = {John Wiley \& Sons, Inc., New York},
      YEAR = {1994},
     PAGES = {xiv+813},
      ISBN = {0-471-05059-8},
   MRCLASS = {14-01},
  MRNUMBER = {1288523},
       DOI = {10.1002/9781118032527},
       URL = {https://doi.org/10.1002/9781118032527},
}

@incollection {pachterStumrfels,
    AUTHOR = {Pachter, Lior and Sturmfels, Bernd},
     TITLE = {Biology},
 BOOKTITLE = {Algebraic statistics for computational biology},
     PAGES = {125--159},
 PUBLISHER = {Cambridge Univ. Press, New York},
      YEAR = {2005},
   MRCLASS = {62P10 (62-07 62F10 92B05)},
  MRNUMBER = {2205869}
}

@article {bcs,
    AUTHOR = {Brustenga i Moncus\'{i}, Laura and Craciun, Gheorghe and Sorea,
              Miruna-{\c S}tefana},
     TITLE = {Disguised toric dynamical systems},
   JOURNAL = {J. Pure Appl. Algebra},
  FJOURNAL = {Journal of Pure and Applied Algebra},
    VOLUME = {226},
      YEAR = {2022},
    NUMBER = {8},
     PAGES = {Paper No. 107035, 24},
      ISSN = {0022-4049},
   MRCLASS = {34D23 (05C90 14P05 14P10 14Q30 34C08 37E99)},
  MRNUMBER = {4379914},
       DOI = {10.1016/j.jpaa.2022.107035},
       URL = {https://doi.org/10.1016/j.jpaa.2022.107035},
}

@book {MR2680546,
    AUTHOR = {Guillemin, Victor and Pollack, Alan},
     TITLE = {Differential topology},
      NOTE = {Reprint of the 1974 original},
 PUBLISHER = {AMS Chelsea Publishing, Providence, RI},
      YEAR = {2010},
     PAGES = {xviii+224},
      ISBN = {978-0-8218-5193-7},
   MRCLASS = {58-01 (57-01)},
  MRNUMBER = {2680546},
       DOI = {10.1090/chel/370},
       URL = {https://doi.org/10.1090/chel/370},
}

@article{haque2022disguised,
  title={The disguised toric locus and affine equivalence of reaction networks},
  author={Haque, Sabina J. and Satriano, Matthew and Sorea, Miruna-{\c{S}}tefana and Yu, Polly Y.},
  journal={SIAM Journal on Applied Dynamical Systems},
  volume={22},
  number={2},
  pages={1423--1444},
  year={2023},
  publisher={SIAM},
DOI={10.1137/22M149853X},
URL={https://epubs.siam.org/doi/abs/10.1137/22M149853X?journalCode=sjaday},
}

@article {MR3920470,
    AUTHOR = {Craciun, Gheorghe},
     TITLE = {Polynomial dynamical systems, reaction networks, and toric
              differential inclusions},
   JOURNAL = {SIAM J. Appl. Algebra Geom.},
  FJOURNAL = {SIAM Journal on Applied Algebra and Geometry},
    VOLUME = {3},
      YEAR = {2019},
    NUMBER = {1},
     PAGES = {87--106},
   MRCLASS = {37C60 (14M25 14T05 34A34 34A60 37A60 37N25 52B20)},
  MRNUMBER = {3920470},
       DOI = {10.1137/17M1129076},
       URL = {https://doi.org/10.1137/17M1129076},
}

@book {feinberg,
    AUTHOR = {Feinberg, Martin},
     TITLE = {Foundations of chemical reaction network theory},
    SERIES = {Applied Mathematical Sciences},
    VOLUME = {202},
 PUBLISHER = {Springer, Cham},
      YEAR = {2019},
     PAGES = {xxix+454},
}

@article {cdss2009,
    AUTHOR = {Craciun, Gheorghe and Dickenstein, Alicia and Shiu, Anne and
              Sturmfels, Bernd},
     TITLE = {Toric dynamical systems},
   JOURNAL = {J. Symbolic Comput.},
  FJOURNAL = {Journal of Symbolic Computation},
    VOLUME = {44},
      YEAR = {2009},
    NUMBER = {11},
     PAGES = {1551--1565},
      ISSN = {0747-7171},
MRREVIEWER = {John B. Little},
       DOI = {10.1016/j.jsc.2008.08.006},
       URL = {https://doi.org/10.1016/j.jsc.2008.08.006},
}

@book {leeTopological,
    AUTHOR = {Lee, John M.},
     TITLE = {Introduction to topological manifolds},
    SERIES = {Graduate Texts in Mathematics},
    VOLUME = {202},
   EDITION = {Second},
 PUBLISHER = {Springer, New York},
      YEAR = {2011},
     PAGES = {xviii+433},
      ISBN = {978-1-4419-7939-1},
   MRCLASS = {57-01 (54-01 55-01)},
  MRNUMBER = {2766102},
       DOI = {10.1007/978-1-4419-7940-7},
       URL = {https://doi.org/10.1007/978-1-4419-7940-7},
}

@article {cy,
 AUTHOR = {Yu, Polly Y. and Craciun, Gheorghe},
     TITLE = {Mathematical Analysis of Chemical Reaction Systems},
      JOURNAL={Israel Journal of Chemistry, 58, 2018},
      URL = {https://doi.org/10.1002/ijch.201800003},
}

@article {HJ,
    AUTHOR = {Horn, Fritz and Jackson, Roy},
     TITLE = {General mass action kinetics},
   JOURNAL = {Arch. Rational Mech. Anal.},
  FJOURNAL = {Archive for Rational Mechanics and Analysis},
    VOLUME = {47},
      YEAR = {1972},
     PAGES = {81--116},
       URL = {https://doi.org/10.1007/BF00251225},
}

@article {Anderson,
    AUTHOR = {Anderson, David F.},
     TITLE = {A proof of the Global Attractor Conjecture in the single linkage class case},
   JOURNAL = {SIAM Journal of Applied Mathematics},
  FJOURNAL = {SIAM Journal of Applied Mathematics},
    VOLUME = {},
      YEAR = {2011},
     PAGES = {},
       URL = {}
}

@article{Brunner_Craciun_2018,
author = {Brunner, James D. and Craciun, Gheorghe},
title = {Robust Persistence and Permanence of Polynomial and Power Law Dynamical Systems},
journal = {SIAM Journal on Applied Mathematics},
volume = {78},
number = {2},
pages = {801-825},
year = {2018},
doi = {10.1137/17M1133762}
}

@article{Craciun_Nazarov_Pantea_2013,
author = {Craciun, Gheorghe and Nazarov, Fedor and Pantea, Casian},
title = {Persistence and Permanence of Mass-Action and Power-Law Dynamical Systems},
journal = {SIAM Journal on Applied Mathematics},
volume = {73},
number = {1},
pages = {305-329},
year = {2013},
doi = {10.1137/100812355}
}

@article{Gopalkrishnan_Miller_Shiu,
author = {Gopalkrishnan, Manoj and Miller, Ezra and Shiu, Anne},
title = {A Geometric Approach to the Global Attractor Conjecture},
journal = {SIAM Journal on Applied Dynamical Systems},
volume = {13},
number = {2},
pages = {758-797},
year = {2014},
doi = {10.1137/130928170}
}

@article{Anderson_Craciun_Kurtz_2010,
author = {Anderson, David F. and Craciun, Gheorghe and Kurtz, Thomas G.},
title = {Product-Form Stationary Distributions for Deficiency Zero Chemical Reaction Networks},
journal = {Bulletin of Mathematical Biology},
volume = {72},
number = {8},
pages = {1947--1970},
year = {2010},
doi = {10.1007/s11538-010-9517-4}
}

@article{Desvillettes_Fellner_Tang_2017,
author = {Desvillettes, Laurent and Fellner, Klemens and Tang, Bao Quoc},
title = {Trend to Equilibrium for Reaction-Diffusion Systems Arising from Complex Balanced Chemical Reaction Networks},
journal = {SIAM Journal on Mathematical Analysis},
volume = {49},
number = {4},
pages = {2666-2709},
year = {2017},
doi = {10.1137/16M1073935}
}

@article{Craciun_Jin_Pantea_Tudorascu_2021,
author = {Gheorghe Craciun and Jiaxin Jin and Casian Pantea and Adrian Tudorascu},
title = {Convergence to the complex balanced equilibrium for some chemical reaction-diffusion systems with boundary equilibria},
journal = {Discrete and Continuous Dynamical Systems - B},
volume = {26},
number = {3},
pages = {1305-1335},
year = {2021},
doi = {10.3934/dcdsb.2020164}
}

\bigskip \medskip

\noindent
\footnotesize 

{\bf \noindent Authors:}

\bigskip

\noindent Gheorghe Craciun\\
University of Wisconsin-Madison, USA\\
 {\tt craciun@math.wisc.edu}
\vspace{0.5cm}

\noindent Jiaxin Jin\\ The Ohio State University, USA\\
{\tt jin.1307@osu.edu}
\vspace{0.5cm}

\noindent Miruna-\c Stefana Sorea\\ SISSA (Scuola Internazionale Superiore di Studi Avanzati), Trieste, Italy and Lucian Blaga University, Sibiu, Romania\\
{\tt msorea@sissa.it}

\end{document}